\documentclass[a4paper]{article}

\usepackage[latin1]{inputenc}  
\usepackage{graphicx}
\usepackage[pdftex,bookmarks]{hyperref}
\usepackage{amsmath}    
\usepackage{amssymb}
\usepackage{amsmath,amsthm}
\usepackage{textcomp}
\usepackage[all]{xy}

\newtheorem{teo}{Theorem}
\newtheorem{defi}{Definition}
\newtheorem{lemma}{Lemma}

\newtheorem{cor}{Corollary}
\newtheorem{prop}{Proposition}

\newtheorem{rem} {Remark}

\DeclareMathOperator{\dif}{Diff}
\DeclareMathOperator{\Tr}{Tr}
\DeclareMathOperator{\tr}{tr}
\DeclareMathOperator{\sgn}{sgn}
\DeclareMathOperator{\spe}{spec}
\DeclareMathOperator{\omo}{Hom}

\title{\huge The $L^2-$Atiyah-Bott-Lefschetz theorem on manifolds with conical singularities. A heat kernel approach.}
\author{Francesco Bei \bigskip \\
 Dipartimento di Matematica, La Sapienza Universit\`a di Roma\\ E-mail addresses: bei@mat.uniroma1.it,\    francescobei27@gmail.com}
\date{}

\begin{document}

\maketitle

\bigskip

\begin{abstract}
Using an approach based on the heat kernel we prove an Atiyah-Bott-Lefschetz  theorem for the $L^2-$Lefschetz numbers associated to an elliptic complex of cone differential operators over a compact manifold with conical singularities. We then apply our results to the case of the de Rham complex.
\end{abstract}
\vspace{1 cm}

\noindent\textbf{Keywords}: Atiyah-Bott-Lefschetz theorem, elliptic complexes, differential cone operators, heat kernel,  geometric endomorphisms,  manifolds with conical singularities.\vspace{1 cm}

\noindent\textbf{Mathematics subject classification}:  58J10, 14F40, 14F43.

\section*{Introduction}
The Atiyah-Bott-Lefschetz theorem for elliptic complexes, see \cite{ABL1}, is a landmark of elliptic theory on closed manifold. After its publication in 1969, several papers have been devoted to this theorem, to explore its applications, to investigate new approaches to its proof and to find some generalizations. For example in \cite{ABL2} the authors use their first paper to explore  applications to the classical elliptic complexes arising in differential geometry;  in \cite{JMB}, \cite{PG}, \cite{KO}, \cite{LZ} and  \cite{P}   the  heat kernel approach is developed, while in  \cite{JB}  an approach using probabilistic methods is employed. In \cite{BRS},  \cite{NSSS}, \cite{NS} \cite{SMS},\cite{MS},  \cite{DT} and \cite{TOTO} the Atiyah-Bott-Lefschetz theorem is extended to some kind of manifolds  that are not closed: for example  \cite{NSSS}  is devoted to the case of elliptic conic operators on manifold with conical singularities, in \cite{SMS} the case of a  manifold with cylindrical ends is studied and in \cite{MS} the case of a complex of Hecke operators over an arithmetic variety is studied.  In particular the use of the heat kernel turned out to be a powerful tool in order to get alternative proofs and extensions of the theorem. Since the heat kernel associated to a conic operator has been intensively studied in the last thirty years, e.g. \cite{BS}, \cite{BS2} \cite{BSE},  \cite{BSE2}, \cite{JC},\cite{MAL} and \cite{EMO},   it is interesting  to  explore its applications  in this context as well, that is to prove an  Atiyah-Bott-Lefschetz theorem over a manifold with conical singularities using the heat kernel.
This is precisely the goal of this paper. \\Our geometric framework is the following: given a compact and orientable manifold with isolated conical singularities $X$,  we consider over its regular part, $reg(X)$ (usually labeled $M$), a complex of elliptic conic differential operators:
\begin{equation}
\label{pergola}
0\rightarrow C^{\infty}_{c}(M,E_{0})\stackrel{P_{0}}{\rightarrow}C^{\infty}_{c}(M,E_{1})\stackrel{P_{1}}{\rightarrow}...\stackrel{P_{n-1}}{\rightarrow}C^{\infty}_{c}(M,E_n)\stackrel{P_{n}}{\rightarrow}0
\end{equation} 
and a geometric endomorphism $T=(T_0,...,T_n)$ of the complex, that is for each $i=0,..,n$,  $T_i=\phi_i\circ f^*$ where $f:X\rightarrow X$ is an isomorphism and $\phi_i:f^*E_i\rightarrow  E_i$ is a bundle homomorphism. Using a conic metric over $M$ we associate to \eqref{pergola}  two Hilbert complexes $(L^2(M,E_i), P_{max/min,i})$ and then, as first step,  we recall the following important property:
\begin{itemize}
\item The cohomology groups of $(L^2(M,E_i), P_{max/min,i})$ are finite dimensional.
\end{itemize}
This result  follows directly from the Fredholm property of elliptic cone operators, see \cite{MAL} Prop. 1.3.16 or \cite{JIM} Prop. 3.14.\\ Afterwards     assuming that $f$ satisfies the following  condition: 
 $$f: \overline{M}\rightarrow \overline{M}\ \text{is\ a\ diffeomorphism}$$ (where $\overline{M}$ is a manifold with boundary which desingularizes  $X$, see Prop. \ref{palcano}, and $f$ is supposed to admit an extension on the whole $\overline{M}$),
 we prove that: 
\begin{itemize}
\item  Each $T_i$ extends to a bounded map acting on $L^2(M,E_i)$ such that $(T_{i+1}\circ P_{max/min,i})(s)=(P_{max/min,i}\circ T_i)(s)$ for each $s\in \mathcal{D}(P_{max/min,,i}).$
\end{itemize}
In this way we can associate to $T$ and \eqref{pergola} two $L^2-$Lefschetz numbers $L_{2,max/min}(T)$ defined as 
\begin{equation}
\label{lapezza}
L_{2,max/min}(T):=\sum_{i=0}^n(-1)^i\Tr(T_i^*:H^i_{2,max/min}(M,E_i)\rightarrow H^i_{2,max/min}(M,E_i))
\end{equation}
Subsequently, using the operators $\mathcal{P}_{i}:=P_{i}^t\circ P_{i}+P_{i-1}\circ P_{i-1}^t$, its absolute and relative extension
and the fact that respective heat operators $e^{-t\mathcal{P}_{abs/rel,i}}:L^2(M,E_i)\rightarrow L^2(M,E_i)$ are trace-class operators we prove the following results:
\begin{itemize}
\item $L_{2,max/min}(T)=\sum_{i=0}^n(-1)^i\Tr(T_i\circ e^{-t\mathcal{P}_{abs/rel,i}})$ for every $t>0$. 
\end{itemize}

After this, to improve the above formula, we require some particular properties about $f$; more precisely we require that:
\begin{itemize}
\item  $f$ fixes each singular point of  $X$. 
\item $Fix(f)$, the fixed points of $f$, is made only by \emph{simple fixed points}.
\end{itemize}
The second requirement means that if $f(q)=q$ and $q\in M$  then the diagonal of $M\times M$ is transverse to the graph on $f$ in $(q,q)$ while if $f(q)=q$ and $q\in sing(X)$ then it  means the following:  over  a neighborhood $U_q $ of $q$, $U_q\cong C_2(L_q)$ the cone over $L_q$,  $f$ takes clearly the form
\begin{equation}
\label{cerronew}
f(r,p)=(rA(r,p),B(r,p)).
\end{equation}
(We make the additional assumption  that  $A(r,p):[0,1)\times L_q\rightarrow [0,1)$ and $B(r,p):[0,1)\times L_q\rightarrow L_q$ are smooth up to zero).   Then we will say that the fixed point is a \textbf{simple} fixed point if for each $p\in L_q$ at least one of the following conditions is satisfied (for more details see Definition \ref{serramaggio}):
\begin{enumerate}
\item   $A(0,p)\neq 1$.
\item  $B(0,p)\neq p$.
\end{enumerate}
Under this conditions, we prove the formula below: $$L_{2,max/min}(T)=\lim_{t\rightarrow 0}(\sum_{q\in Fix(f)}\sum_{i=0}^n(-1)^i\int_{U_q}\tr(\phi_i\circ k_{abs/rel,i}(t,f(x),x))dvol_g)$$ where $\phi_i\circ k_{abs/rel,i}(t,f(x),x)$ is the smooth kernel of $T_i \circ e^{-t\mathcal{P}_{abs/rel},i}$ and $U_q$ is a neighborhood of $q$ (obviously when $q\in sing(X)$ then we mean the regular part of $U_q$).
 Moreover under some additional hypothesis,  in particular that  \eqref{cerronew} modifies in the following way:
\begin{equation}
\label{cerronews}
f(r,p)=(rA(p),B(p))
\end{equation}
we prove the following formulas,  (see Theorem \ref{edimburgo}), which are the main result of the paper  :
\begin{equation}
\label{vercelli}
L_{2,max/min}(T)=\sum_{q\in Fix(f)\cap M}\sum_{i=0}^n\frac{(-1)^i\Tr(\phi_i)}{|det(Id-d_q(f))|}+\sum_{q\in  sing(X)}\sum_{i=0}^n(-1)^i\zeta_{T_i,q}(\mathcal{P}_{abs/rel,i})(0)
\end{equation}
\begin{equation}
\label{stresa}
\zeta_{T_i,q}(\mathcal{P}_{abs/rel,i})(0)=\frac{1}{2\nu}\int_{0}^{\infty}\frac{dx}{x}\int_{L_q} \tr(\phi_i\circ e^{-x\mathcal{P}_{abs/rel,i}}(A(p),B(p),1,p))dvol_h.
\end{equation}

Finally, in the last part of the paper, we apply the previous results to the de Rham complex. We get an analytic construction of the Lefschetz numbers arising in intersection cohomology  and a topological interpretation of the contributions given by the singular points to the $L^2-$Lefschetz numbers. In particular, under suitable conditions,  we prove the following formula:
\begin{equation}
\label{pontesasso}
I^{\underline{m}}L(f)=L_{2,max}(T)=\sum_{q\in Fix(f)\cap reg(X)}\sgn det(Id-d_qf)+
\end{equation}
$$+\sum_{q\in  sing(X)}\sum_{i<\frac{m+1}{2}}(-1)^i\Tr(B^*:H^i(L_q)\rightarrow H^i(L_q))$$
where $I^{\underline{m}}L(f)$ is the intersection Lefschetz number arising in intersection cohomology, $T$ is the endomorphism of $(L^2\Omega^i(M,g), d_{max,i})$ induced by $f$ and  $B$ is the diffeomorphism of the link $L_q$ such that, in a neighborhood of $q$,  $f$ satisfies \eqref{cerronews}. In particular  from \eqref{pontesasso} we get:
\begin{equation}
\label{mogadiscio}
\sum_{i=0}^{m+1}(-1)^i\zeta_{T_i,q}(\Delta_{abs,i})(0)=\sum_{i<\frac{m+1}{2}}(-1)^i\Tr(B^*:H^i(L_q)\rightarrow H^i(L_q)).
\end{equation}
\vspace{1 cm}

\noindent As recalled at the beginning of the introduction also \cite{NSSS} is devoted to the Atiyah-Bott-Lefschetz theorem on manifold with conical singularities. Anyway there are some substantial differences between our paper and \cite{NSSS}: the notion of ellipticity used there, which is taken from \cite{BWS}, is stronger than that one used in this paper; in particular  the de Rham complex is not elliptic for the definition given in  \cite{BWS}. Moreover    the complexes  considered in \cite{NSSS} are complexes of weighted Sobolev space while our complexes are Hilbert complexes of unbounded operator defined on some natural extensions of their core domain; finally also the techniques used are different because we use the heat kernel while in \cite{NSSS}  the existence of a  parametrix of an elliptic cone operator  is used.
Some results of this paper are also close to  results proved in \cite{MAL}: indeed in \cite{MAL} the heat kernel is studied in an equivariant situation and an equivariant index theorem is proved (see  Corollary 2.4.7 ). Also in this case there are some relevant differences: the Lie group $G$ acting in \cite{MAL} is a {\em compact} Lie group of isometry, while in our work we just require  that the map $f$ is a diffeomorphism. Moreover the non degeneracy conditions that we require on the fixed point of $f$ led us to different formulas to those  stated in \cite{MAL}.  On the other hand, for the geometric endomorphisms  considered in \cite{MAL}, that is those  induced by  isometries $g$ lying in a compact Lie group $G$, the formula obtained by Lesch applies to a more general  case than the ours because in his work there are not assumptions on the fixed points set while in our work there are. \\
 Moreover, as recalled above, the last part of this paper contains several applications to the de Rham complex which are not mentioned in the other papers. \vspace{1 cm}

\textbf{Acknowledgment}.  I wish to thank Paolo Piazza for having suggested this subject, for  his help and for many helpful discussions. I wish also to thank Pierre Albin for having invited me to spend the months of March and April 2012 at  the University of Illinois at Urbana-Champaign and  for many interesting discussions.

\section{Background}
\subsection{Hilbert complexes}

In this first subsection we recall briefly the notion of Hilbert complex and how it appears in riemannian geometry. We refer to \cite{BL} for a thorough discussion about this subject.
\begin{defi} A Hilbert complex is a complex, $(H_{*},D_{*})$ of the form:
\begin{equation}
\label{mm}
0\rightarrow H_{0}\stackrel{D_{0}}{\rightarrow}H_{1}\stackrel{D_{1}}{\rightarrow}H_{2}\stackrel{D_{2}}{\rightarrow}...\stackrel{D_{n-1}}{\rightarrow}H_{n}\rightarrow 0,
\end{equation}
where each $H_{i}$ is a separable Hilbert space and each map $D_{i}$ is a closed operator called the differential such that:
\begin{enumerate}
\item $\mathcal{D}(D_{i})$, the domain of $D_{i}$, is dense in $H_{i}$.
\item $ran(D_{i})\subset \mathcal{D}(D_{i+1})$.
\item $D_{i+1}\circ D_{i}=0$ for all $i$.
\end{enumerate}
\end{defi}
The cohomology groups of the complex are $H^{i}(H_{*},D_{*}):=Ker(D_{i})/ran(D_{i-1})$. If the groups $H^{i}(H_{*},D_{*})$ are all finite dimensional we say that it is a  $Fredholm\ complex$. 

Given a Hilbert complex there is a dual Hilbert complex
\begin{equation}
0\leftarrow H_{0}\stackrel{D_{0}^{*}}{\leftarrow}H_{1}\stackrel{D_{1}^{*}}{\leftarrow}H_{2}\stackrel{D_{2}^{*}}{\leftarrow}...\stackrel{D_{n-1}^{*}}{\leftarrow}H_{n}\leftarrow 0,
\label{mmp}
\end{equation}
defined using $D_{i}^{*}:H_{i+1}\rightarrow H_{i}$, the Hilbert space adjoints of the differentials\\ $D_{i}:H_{i}\rightarrow H_{i+1}$. The cohomology groups of $(H_{j},(D_{j})^*)$, the dual Hilbert  complex, are $$H^{i}(H_{j},(D_{j})^*):=Ker(D_{n-i-1}^{*})/ran(D_{n-i}^*).$$\\For all $i$ there is also a laplacian $\Delta_{i}=D_{i}^{*}D_{i}+D_{i-1}D_{i-1}^{*}$ which is a self-adjoint operator on $H_{i}$ with domain 
\begin{equation}
\label{saed}
\mathcal{D}(\Delta_{i})=\{v\in \mathcal{D}(D_{i})\cap \mathcal{D}(D_{i-1}^{*}): D_{i}v\in \mathcal{D}(D_{i}^{*}), D_{i-1}^{*}v\in \mathcal{D}(D_{i-1})\}
\end{equation} and nullspace: 
\begin{equation}
\label{said}
\mathcal{H}^{i}(H_{*},D_{*}):=ker(\Delta_{i})=Ker(D_{i})\cap Ker(D_{i-1}^{*}).
\end{equation}

The following propositions are  standard results for these complexes. The first result is a weak Kodaira decomposition:

\begin{prop}
\label{beibei}
 [\cite{BL}, Lemma 2.1] Let $(H_{i},D_{i})$ be a Hilbert complex and $(H_{i},(D_{i})^{*})$ its dual complex, then: $$H_{i}=\mathcal{H}^{i}\oplus\overline{ran(D_{i-1})}\oplus\overline{ran(D_{i}^{*})}.$$
\label{kio}
\end{prop}

The reduced cohomology groups of the complex are: $$\overline{H}^{i}(H_{*},D_{*}):=Ker(D_{i})/(\overline{ran(D_{i-1})}).$$
By the above proposition there is a pair of  weak de Rham isomorphism theorems:
 \begin{equation}
\left\{
\begin{array}{ll}
\mathcal{H}^{i}(H_{*},D_{*})\cong\overline{H}^{i}(H_{*},D_{*})\\
\mathcal{H}^{i}(H_{*},D_{*})\cong\overline{H}^{n-i}(H_{*},(D_{*})^{*})
\end{array}
\right.
\label{pppp}
\end{equation}
where in the second case we mean the cohomology of the dual Hilbert complex.\\
The complex $(H_{*},D_{*})$ is said $ weak\  Fredholm$  if $\mathcal{H}_{i}(H_{*},D_{*})$ is finite dimensional for each $i$. By the next propositions it follows immediately that each Fredholm complex is a weak Fredholm  complex.

\begin{prop}
\label{topoamotore}
[\cite{BL}, Corollary 2.5] If the cohomology of a Hilbert complex $(H_{*}, D_{*})$ is finite dimensional then, for all $i$,  $ran(D_{i-1})$ is closed  and  $H^{i}(H_{*},D_{*})\cong \mathcal{H}^{i}(H_{*},D_{*}).$
\end{prop}

\begin{prop}[\cite{BL}, Corollary 2.6] A Hilbert complex  $(H_{j},D_{j}),\ j=0,...,n$ is a Fredholm complex (weak Fredholm) if and only if  its dual complex, $(H_{j},D_{j}^{*})$, is Fredholm (weak Fredholm). If it is Fredholm then
\begin{equation}
\label{rabat}
\mathcal{H}_{i}(H_{j},D_{j})\cong H_{i}(H_{j},D_{j})\cong H_{n-i}(H_{j},(D_{j})^{*})\cong \mathcal{H}_{n-i}(H_{j},(D_{j})^{*}).
\end{equation}
Analogously in the  the weak Fredholm case we have:
\begin{equation}
\mathcal{H}_{i}(H_{j},D_{j})\cong\overline{ H}_{i}(H_{j},D_{j})\cong\overline{ H}_{n-i}(H_{j},(D_{j})^{*})\cong \mathcal{H}_{n-i}(H_{j},(D_{j})^{*}).
\end{equation}
\label{fred}
\end{prop}

\begin{prop}
\label{fonteavellana}
 A Hilbert complex  $(H_{j},D_{j}),\ j=0,...,n$ is a Fredholm complex if and only if  for each $i$ the operator $\Delta_{i}$ defined in \eqref{saed} is a Fredholm operator on its domain endowed with the graph norm. 
\end{prop}

\begin{proof}
See \cite{BWS}, Lemma 1 pag 203. 
\end{proof}

Now we recall another result which  shows that it is possible to compute the cohomology groups of an Hilbert complex using a core subcomplex $$\mathcal{D}^{\infty}(H_{i})\subset H_{i}.$$ For all $i$ we define  $\mathcal{D}^{\infty}(H_{i})$ as consisting  of all elements $\eta$ that are in the domain of $\Delta_{i}^{l}$ for all $l\geq 0.$

\begin{prop}[\cite{BL}, Theorem 2.12] The complex $(\mathcal{D}^{\infty}(H_{i}),D_{i})$ is a subcomplex quasi-isomorphic to the complex $(H_{i},D_{i})$
\label{dfff}
\end{prop}

As it is well known, riemannian geometry offers a framework in which Hilbert and (sometimes) Fredholm complexes can be built in a natural way.  The rest of this subsection is devoted to recall  these constructions.\\
Let  $(M,g)$ be  an open and oriented riemannian manifold of dimension $m$ and let $E_{0},...,E_{n}$ be vector bundles over $M$. For each $i=0,...,n$ let $C^{\infty}_{c}(M,E_{i})$ be the space of smooth section with compact support. If we put on each vector bundle a metric $h_{i}\ i=0,...,n$ the we can construct in a natural way a sequences of Hilbert space $L^{2}(M,E_{i}),\ i=0,...,n$ as the completion of $C^{\infty}_{c}(M,E_{i}).$ Now suppose that we have a complex of differential operators :
\begin{equation}
0\rightarrow C^{\infty}_{c}(M,E_{0})\stackrel{P_{0}}{\rightarrow}C^{\infty}_{c}(M,E_{1})\stackrel{P_{1}}{\rightarrow}C^{\infty}_{c}(M,E_{2})\stackrel{P_{2}}{\rightarrow}...\stackrel{P_{n-1}}{\rightarrow}C^{\infty}_{c}(M,E_{n})\rightarrow 0,
\label{yygg}
\end{equation}
To turn this complex into a Hilbert complex we must specify a closed extension of $P_{*}$ that is an operator between $L^{2}(M,E_{*})$ and  $L^{2}(M,E_{*+1})$ with closed graph which is an extension of $P_{*}$.   We start recalling  the two canonical closed extensions of $P$.

\begin{defi} The maximal extension $P_{max}$; this is the operator acting on the domain:
\begin{equation} 
\mathcal{D}(P_{max,i})=\{\omega\in L^{2}(M,E_{i}): \exists\ \eta\in L^{2}(M,E_{i+1})
\end{equation}

$$ s.t.\ <\omega,P^t_{i}\zeta>_{L^{2}(M,E_{i})}=<\eta,\zeta>_{L^{2}(M,E_{i+1})}\ \forall\ \zeta\in C_{0}^{\infty}(M,E_{i+1}) \}$$ where $P^t_i$ is the formal adjoint of $P_i$.

In this case $P_{max,i}\omega=\eta.$ In  other words $\mathcal{D}(P_{max,i})$ is the largest set of forms $\omega\in L^{2}(M, E_{i})$ such that $P_{i}\omega$, computed distributionally, is also in $L^{2}(M,E_{i+1}).$
\label{nino}
\end{defi}

\begin{defi} The minimal extension $P_{min,i}$; this is given by the graph closure of $P_{i}$ on $C_{0}^{\infty}(M, E_{i})$ respect to the norm of $L^{2}(M,E_{i})$, that is,
\begin{equation} \mathcal{D}(P_{min,i})=\{\omega\in L^{2}(M,E_{i}): \exists\ \{\omega_{j}\}_{j\in J}\subset C^{\infty}_{0}(M,E_{i}),\ \omega_{j}\rightarrow \omega ,\       P_{i}\omega_{j}\rightarrow \eta\in L^{2}(M,E_{i+1})\} 
\end{equation}
and in this case $P_{min,i}\omega=\eta$
\label{reso}
\end{defi}

Obviously $\mathcal{D}(P_{min,i})\subset \mathcal{D}(P_{max,i})$. Furthermore, from these definitions, it follows immediately that $$P_{min,i}(\mathcal{D}(P_{min,i}))\subset \mathcal{D}(P_{min,i+1}),\ P_{min,i+1}\circ P_{min,i}=0$$ and that $$P_{max,i}(\mathcal{D}(P_{max,i}))\subset \mathcal{D}(P_{max,i+1}),\ P_{max,i+1}\circ P_{max,i}=0.$$ \\Therefore $(L^{2}(M,E_{*}),P_{max/min,*})$ are both Hilbert complexes and their cohomology groups, respectively reduced cohomology groups, are denoted respectively by $H_{2,max/min}^{i}(M, E_{*})$ and  $\overline{H}_{2,max/min}^{i}(M,E_{*})$. 

Another straightforward but important fact is that the Hilbert complex adjoint of \\$(L^{2}(M,E_{*}),P_{max/min,*})$ is $(L^{2}(M,E_{*}),P^t_{min/max,*})$, that is
\begin{equation}
(P_{max,i})^{*}=P^t_{min,i},\     (P_{min,i})^{*}=P^t_{max,i}.
\end{equation}

Using  Proposition \ref{kio}  we obtain two weak Kodaira decompositions:
\begin{equation}
L^{2}(M,E_{i})=\mathcal{H}^{i}_{abs/rel}(M,E_i)\oplus \overline{ran(P_{max/min,i-1})}\oplus \overline{ran(P^t_{min/max,i})}
\label{kd}
\end{equation}
with summands mutually orthogonal in each case. For the first summand on the right, called the absolute or relative Hodge cohomology, we have by \eqref{said}:
\begin{equation}
\label{nannabobo}
\mathcal{H}^{i}_{abs/rel}(M,E_*)=Ker(P_{max/min,i})\cap Ker(P^t_{min/max,i-1}).
\end{equation}
We can also consider the two natural laplacians associated to these Hilbert complexes, that is for each $i$
\begin{equation}
\label{kaak}
\mathcal{P}_{abs,i}:=P_{min,i}^t\circ P_{max,i}+P_{max,i-1}\circ P_{min,i-1}^t
\end{equation}
and 
\begin{equation}
\label{kkll}
\mathcal{P}_{rel,i}:=P_{max,i}^t\circ P_{min,i}+P_{min,i-1}\circ P_{max,i-1}^t
\end{equation}
with domain described in \eqref{saed}. Using \eqref{said} and \eqref{pppp} it follows that the nullspace of  \eqref{kaak} is isomorphic to the absolute Hodge cohomology which is in turn isomorphic to the reduced  cohomology of the Hilbert complex $(L^{2}(M,E_{*}),P_{max,*})$. Analogously,  using again \eqref{said} and \eqref{pppp}, it follows that the nullspace of  \eqref{kkll} is isomorphic to the relative Hodge cohomology which is in turn isomorphic to the reduced  cohomology of the Hilbert complex $(L^{2}(M,E_{*}),P_{min,*})$.\\
Finally we recall that we can define other two Hodge cohomology groups $\mathcal{H}^i_{max/min}(M,E_*)$ defined as
\begin{equation}
\label{aaaaaaaa}
\mathcal{H}^i_{max/min}(M,E_*)=Ker(P_{max/min,i})\cap Ker(P^t_{max/min,i-1}).
\end{equation}

\subsection{Manifolds with conical singularities and differential  cone operators}

\begin{defi}
Let $L$ an open manifold. The cone over $L$, usually labeled $C(L)$, is the topological space defined as
\begin{equation}
\label{chiaserna}
L\times [0,\infty)/(\{0\}\times L).
\end{equation}
The truncated cone, usually labeled $C_{a}(L)$, is defined as 
\begin{equation}
\label{moria}
L\times [0,a)/(\{0\}\times L).
\end{equation}
Finally with $\overline{C_a(L)}$ we mean
\begin{equation}
\label{fossato}
L\times [0,a]/(\{0\}\times L).
\end{equation}
In both the above cases, with $v$, we will label the vertex of the cone or the  truncated cone, that is $C(L)-(L\times (0,\infty))$, $C_{a}(L)-(L\times(0,a))$ and $\overline{C_{a}(L)}-(L\times(0,a])$ respectively.
\end{defi}

\begin{defi}
\label{gubbio}
A manifold with conical singularities $X$ is a metrizable, locally compact,  Hausdorff space such that there exists a sequence of points $\{p_{1},...,p_{n},...\}\subset X$ which satisfies the following properties:
\begin{enumerate}
\item $X-\{p_{1},...,p_{n},...\}$ is a smooth open manifold.
\item For each $p_{i}$ there exist an open neighborhood $U_{p_i}$, a closed manifold $L_{p_i}$ and a  map $\chi_{p_i}:U_{p_i}\rightarrow C_{2}(L_{p_i})$ such that $\chi_{p_i}(p_i)=v$ and $ \chi_{p_i}|_{U_{p_i}-\{p_{i}\}}:U_{p_i}-\{p_i\}\rightarrow L_{p_i}\times (0,2)$  is a diffeomorphism. 
\end{enumerate}
\end{defi}

The regular and the singular part of $X$ are defined as $$sing(X)=\{p_1,...,p_n,...\},\ reg(X):=X-sing(X)=X-\{p_1,...,p_n,...\}.$$The singular points $p_i$ are usually called \emph{conical points} and the smooth closed manifold $L_{p_i}$ is usually called the \emph{link} relative to the point $p_i$. If $X$  is compact then it is clear, from the above definition,  that the sequences of conical points $ \{p_1,...,p_n,...\}$ is made of isolated points  and therefore on $X$ there are just a finite number of conical points.\\A manifold with conical singularities is a particular case of a compact smoothly stratified pseudomanifold; more precisely it is a compact smoothly stratified pseudomanifold with depth 1 and with the singular set made of a sequence of isolated points. Since in this  paper we will work exclusively with compact manifolds with conical singularities we prefer to omit the definition of smoothly compact stratified pseudomanifold and the notions related to it and refer to \cite{ALMP} for a thorough discussion on this subject.\\
\begin{rem}
\label{peschiera}
Let $X$ be a compact manifold with one conical singularity $p$ and let $L_{p}$ its link; it follows from Definition \ref{gubbio} that we can decompose $X$ as $$X\cong \overline{Y}\cup_{L_{p}}\overline{C_{1}(L_{p})}$$ where $\overline{Y}$ is a compact manifold with boundary defined as $X-\chi_{p}^{-1}(C_{1}(L_p))$.  Obviously this decomposition generalizes in a natural way when $X$ has several conical points. As we will see in one of the following sections this decomposition is the starting point to study the heat kernel on $X$ and we will use it to calculate the contribution given by the conical points to the Lefschetz number of some geometric endomorphisms.
\end{rem}

Now we recall from \cite{ALMP} a particular case, which is suitable for our purpose, of an important result which  describe a blowup process  to resolve the singularities of a compact smoothly stratified pseudomanifold.

\begin{prop} 
\label{palcano}
Let $X$ be a compact manifold with conical singularities. Then there exists a manifold with boundary $\overline{M}$ and a blow-down map $\beta:\overline{M}\rightarrow X$ which has the following properties: 
\begin{enumerate}
\item $\beta|_{M}:M\rightarrow reg(X)$, where $M$ is the interior of $\overline{M}$, is a diffeomorphsim.
\item If $N$ is a connected component of $\partial \overline{M}$ and if  $U\cong  N\times [0,1)$ is a collar  neighborhood of $N$ then $\beta(U)=N\times [0,1)/(N\times \{0\})$. In particular $\beta(N)=p$ where $p$ is a conical point of $X$ and $N$ becomes one of the connected components of  the link of $p$.
\item If for each conical point $p_i$  the relative link $L_{p_i}$ is connected,  then there is a bijection between the conical points of $X$ and the connected components of $\partial \overline{M}.$
\end{enumerate}
\end{prop}

\begin{proof}
See \cite{ALMP}, Proposition 2.5.
\end{proof}

Now we introduce a class of natural riemannian  metrics  on these spaces.

\begin{defi}
Let $X$ be a manifold with conical singularities. A conic metric $g$ on $reg(X)$ is riemannian metric with the following property: for each  conical point $p_i$ there exists a map $\chi_{p_i}$, as defined in Definition \ref{gubbio},  such that \begin{equation}
\label{pianello}
(\phi_{p_i}^{-1})^*(g|_{U_{p_{i}}})=dr^2+r^2h_{L{p_{i}}}(r)
\end{equation}
 where $h_{L{p_{i}}}(r)$ depends smoothly on $r$ up to $0$ and for each fixed $r\in [0,1)$  it is a riemannian metric on $L_{p_{i}}.$ Analogously, if $\overline{M}$ is manifold with boundary and $M$ is its interior part, then $g$ is a conic metric on $M$ if it is a smooth, symmetric section of $T^{*}\overline{M}\otimes  T^*\overline{M}$, degenerate over the boundary, such that over a collar neighborhood $U$ of $\partial \overline{M}$, $g$ satisfies \eqref{pianello} with respect to some diffeomorphism $\chi:U\rightarrow [0,1)\times \partial \overline{M}.$
\end{defi}

The next step is to recall the notion of differential  cone operator and its main properties.  Before to proceed we introduce some notations that we will use steadily through the paper. \\
Given an open manifold $M$ and two vector bundles $E,F$ over it, with $\dif^{n}(M,E,F), n\in \mathbb{N}$, we will label the space of differential operator $P:C^{\infty}_{c}(M,E)\rightarrow C^{\infty}_{c}(M,F)$ of order $n$. Given $\overline{M}$, a manifold with boundary, we will label  with $N$ the boundary of $\overline{M}$ and with $M$ the interior part of $\overline{M}$. Given a vector bundle $E$ over $\overline{M}$, with $E_{N}$ we mean the restriction of $E$ on $N$. Finally each metric $\rho$ over $E$ (riemannian if $E$ is real or hermitian if $E$ is complex) is assumed to be a non degenerate metric up to the boundary. The next definition is taken from \cite{MAL}:
\begin{defi}
\label{iconti}
Let $\overline{M}$ be a manifold with boundary $N=\partial \overline{M}$. Let $E,F$ be two vector bundles on $\overline{M}$. Let $\overline{U}_N$ be a collar neighborhood of $N$, $\overline{U}_N\cong [0,\epsilon)\times N$ and let $U_N=\overline{U}_N-N$.  A  differential cone operator of order $\mu\in \mathbb{N}$ and weight $\nu>0$
 is a differential operator $P:C^{\infty}_c(M,E)\rightarrow C^{\infty}_{c}(M,F)$ such that on $U_N$ it takes the form:
\begin{equation}
\label{genova}
P|_{U_{N}}=x^{-\nu}\sum_{i=0}^{\mu}A_{k}(-x\frac{\partial}{\partial x})^{k}
\end{equation}
  where $A_k\in C^{\infty}([0,\epsilon), \dif^{\mu-k}(N,E_{N},F_{N}))$ and $x$ is  the coordinate on $[0,\epsilon)$ .
As in \cite{MAL} we will label with $\dif_0^{\mu,\nu}(M,E,F)$ the space of differential cone operators between the bundles $E$ and $F$.
\end{defi}
Now we explain what we mean by {\em differential\ cone\ operator}  on a manifold $X$ with conical singularities. In the previous definition we recalled the notion of differential cone operator acting on the smooth sections with compact support of two vector bundles $E,F$ defined on a manifold $\overline{M}$ with boundary. In Proposition \ref{palcano}, given a manifold with conical singularities $X$, we stated the existence of a manifold with boundary $\overline{M}$ endowed with a blow down map $\beta:\overline{M}\rightarrow X$ which desingularizes $X$. Therefore given two vector bundles $E,F$ on $reg(X)$ and $P\in \dif(reg(X),E,F)$ we will say that $P$ is a differential cone operators if the following properties are satisfied:
\begin{enumerate}
\item $\beta^*(E),\beta^*(F)$ that are vector bundles on $M$, the interior of $\overline{M}$, extend as smooth vector bundles over the whole $\overline{M}$. In the same way, if $E$ and $F$ are endowed with metrics $\rho_1$ and $\rho_2$ then $\beta^*\rho_1$ and $\beta^*\rho_2$ extend as non degenerate metric up to the boundary of $\overline{M}.$
\item The  differential operator induced by $P$ through $\beta$ between $C^{\infty}_{c}(M,\beta^*E,\beta^*F)$ is a differential cone operator in the sense of Definition \ref{iconti}. 
\end{enumerate}
In the rest of the paper, with a slight abuse of notation, we will identify $M$ with $reg(X)$, $E$ with $\beta^*E$, $F$ with $\beta^*F$ and  $P$ with the operator that it induces through $\beta$  between $C^{\infty}_{c}(M,\beta^*E,\beta^*F)$. 

\begin{rem}
\label{brolio}
We can reformulate Definition \ref{iconti} in the following way: $P$ is differential cone operator of order $\mu$ and weight $\nu$ if and only if $x^{\nu}P$ is a $b-$differential operator of order $\mu$ in the sense of Melrose. For the definition of $b-$operator and the full development of this subject we refer to the monograph \cite{RBM}. Using this approach we have $\dif_0^{\mu,\nu}(M,E,F)=x^{-\nu}\dif_{b}^{\mu}(M,E,F)$. This last point of view is used for example in \cite{JIM} .
\end{rem}

Now we introduce the notion of ellipticity:
\begin{defi}
\label{londra}
 Let $\overline{M}$ be a manifold with boundary and let $E,F$ be two vector bundles over $\overline{M}$. Let  $P\in \dif_0^{\mu,\nu}(\overline{M},E,F)$ and let $\sigma^{\mu}(P)$ be  its principal symbol. Then $P$ is called elliptic if it is elliptic on $M$ in the usual sense and if 
\begin{equation}
\label{vilano}
x^{\nu}\sigma^{\mu}(P)(x,p,x^{-1}\tau, \xi)
\end{equation}
 is invertible for $(x,p)\in [0,\epsilon)\times N$ and $(\tau,\xi)\in T^*\overline{M}-\{0\}, \xi\in T_p^*(N)$.
\end{defi}
In the above definition there is implicit the natural identification of $T^*\overline{M}|_{[0,\epsilon)\times N}$ with $\mathbb{R}\times T^*N$.

\begin{defi}
\label{parigi}
Let $\overline{M}, E,F$ and $P$ be as in the previous definition. The conormal symbol of $P$, as defined in \cite{MAL},  is the family of differential operators, acting between $C^{\infty}(N,E_{N},F_{N})$, defined as
\begin{equation}
\label{cagli}
\sigma^{\mu,\nu}_{M}(P)(z):= \sum_{k=0}^{\mu}A_{k}(0)z^k
\end{equation}
\end{defi}
Now we make some further comments about the notion of ellipticity introduced in Definition \ref{londra}.
The requirement \eqref{vilano} in Definition \ref{londra} means that  $$\sum_{k=0}^{\mu} \sigma^{\mu-k}(A_{k}(x))((p,\xi))
\sigma^{k}((-x\frac{\partial}{\partial x})^k)(x,x^{-1}\tau)=\sum_{k=0}^{\mu}\sigma^{\mu-k}(A_{k}(x))((p,\xi))(-i\tau)^{k}$$ is invertible. On $M$ this is covered by classical ellipticity and for $x=0$ it is equivalent to require that \eqref{cagli} is a parameter dependent elliptic family of differential operators with parameters in $i\mathbb{R}$. \\Using again the $b$ framework of Melrose, Definition \ref{londra} is equivalent to say that the  $b-$principal symbol of $P':=x^{\nu}P$, that is $\sigma_b^{\mu}(P'):=\sigma^{\mu}(P')(x,p,x^{-1}\tau, \xi)$,  as an object  lying in\\ $C^{\infty}(T_{b}^*\overline{M},\omo(\pi_{b}^*E,\pi_{b}^*F))$, where $\pi_{b}:T_{b}^*\overline{M}\rightarrow \overline{M}$ is the $b-$cotangent bundle of $\overline{M}$, is an isomorphism on $T^*_{b}\overline{M}-\{0\}$. For further details on these approach see \cite{JIM} and the relative bibliography.\\Finally we remark that in Definition \ref{londra} we followed \cite{MAL} and \cite{JIM}. This is slightly different from those given, for example, in \cite{NSSS}, \cite{NS} and \cite{BWS}. The definition given in  these papers, in fact, requires the invertibility of the conormal symbol on a certain weight line (for more details see the above papers). By the fact that we are interested to study the operators on their natural domains, that is the maximal and the minimal one, we can waive this requirement (see \cite{MAL} pag. 13 for more comments about this). \vspace{ 1 cm}  

Finally we conclude this subsection stating an important proposition  on the theory of differential cone operators:
\begin{teo}
\label{gualdo}
Let $(\overline{M},g)$  be a compact and oriented manifold of dimension $m$ with boundary where $g$ is a conic metric over $M$; let $E,F$ be two hermitian vector bundles over $\overline{M}$ and let $P\in \dif_0^{\mu,\nu}(M,E,F)$ be an elliptic differential cone operator. 
\begin{enumerate}
\item Each closed extension  $\overline{P}:L^2(M,E)\rightarrow L^{2}(M,F)$ of $P$ is a Fredholm operator on its domain, $\mathcal{D}(\overline{P})$, endowed with the graph norm. 
\item Suppose that $E=F$ and that $P$ is   positive. Suppose, in addition, that on a collar neighborhood of $\partial \overline{M}$ the metric $\rho$ on $E$ does not depend on $r$ and the conic metric $g$ satisfies $g=dr^2+r^2h$ where $h$ is any riemannian metric over $\partial \overline{M}$ which does not depend on $r$.  Then, for each positive self-adjoint extension  $\overline{P}$ of $P$, the heat operator  $e^{-t\overline{P}}:L^2(M,E)\rightarrow L^2(M,E)$ is a trace-class operator. Moreover $\overline{P}$ is discrete and  the sequences of eigenvalues of $\overline{P}$ satisfies $\lambda_j \sim C j^{\frac{\mu}{m}}$.
\end{enumerate}
\end{teo}

\begin{proof}
For the first statement see \cite{MAL} Prop. 1.3.16 or \cite{JIM} Prop. 3.14. For the second one see \cite{MAL} Theorem 2.4.1 and Corollary 2.4.3. 
\end{proof}

\subsection{Elliptic complex on manifolds with conical singularities}
The aim of this subsection is to define the notion of elliptic complex on a manifold with conical singularities. As for the notion of ellipticity, the definition of elliptic complex on a manifold with conical singularities was introduced in \cite{BWS}, pag. 205, but our definition is slightly different because we waive some requirements about the sequence of conormal symbols on a certain weight line. The reason is still given by the fact that we are interested on the minimal and maximal extension of a complex  differential cone operators.\\ Let $\overline{M}$ be a manifold with boundary, $E_{0},...,E_{n}$ a sequence of vector bundle over $\overline{M}$ and consider $P_{i}\in\dif^{\mu,\nu}_{0}(M,E_{i},E_{i+1})$  such that 
\begin{equation}
\label{pisa}
0\rightarrow C^{\infty}_{c}(M,E_{0})\stackrel{P_{0}}{\rightarrow}C^{\infty}_{c}(M,E_{1})\stackrel{P_{1}}{\rightarrow}...\stackrel{P_{n-1}}{\rightarrow}C^{\infty}_{c}(M,E_n)\stackrel{P_{n}}{\rightarrow}0
\end{equation} is  a complex.
We have the following definition:
\begin{defi}
\label{perugia}
The complex \eqref{pisa} is an elliptic complex if it is an elliptic complex in the usual sense on $M$ and if the sequence
\begin{equation}
\label{pistoia}
0\rightarrow \pi^*E_{0}\rightarrow \pi^*E_1\rightarrow...\rightarrow \pi^*E_n\rightarrow0
\end{equation}
 where the maps are given by $x^{\nu}\sigma^{\mu}(P_i)(x,p,x^{-1}\tau,\xi):\pi_i^*E_i\rightarrow\pi_{i+1}^*E_{i+1}$ is an exact sequence up to $x=0$ over $T^*\overline{M}-\{0\}$.
\end{defi}

With the help of Melrose's $b$ framework we can reformulate the previous definition in the following way:
\eqref{pisa} is an elliptic complex if and only if the following sequence is exact over $T^*_b(\overline{M})-\{0\}$:
\begin{equation}
\label{casalvecchio}
0\rightarrow \pi_{b}^*E_0\stackrel{\sigma_b^{\mu}(P'_{0})}{\rightarrow}\pi_{b}^*E_1\stackrel{\sigma_b^{\mu}(P'_{1})}{\rightarrow}...\stackrel{\sigma_b^{\mu}(P'_{n-1})}{\rightarrow}\pi_{b}^*E_n\stackrel{\sigma_b^{\mu}(P'_{n})}{\rightarrow}0
\end{equation} 
where $P'=x^{\nu}P$, that is the $b-$operator  naturally associated to $P$, $\pi_{b}:T_b^*\overline{M}\rightarrow \overline{M}$ is the $b-$cotangent bundle and $\sigma_b^{\mu}(P'_i)\in C^{\infty}(\overline{M},\omo(\pi_b^*E_i,\pi_b^*E_{i+1}))$ is the $b-$principal symbol of $P_i'$.\\We have the following  proposition:
\begin{prop}
\label{lucca}
Consider a complex of differential cone operators as in \eqref{pisa}. Suppose moreover that $M$ is endowed with a conic metric $g$. Then the complex  is an elliptic complex if and only if for each $i=0,...,n$ $$P^t_{i}\circ P_i+P_{i-1}\circ P_{i-1}^t:C^{\infty}_{c}(M,E_i)\rightarrow C^{\infty}_c(M,E_{i})$$ is an elliptic differential cone operator.
\end{prop}
\begin{proof}
It is clear that if $P\in \dif_0^{\mu,\nu}(M,E_{i},E_{i+1})$ then also $P^t\in \dif_0^{\mu,\nu}(M,E_{i+1},E_i)$ where $P_t:C^{\infty}_{c}(M,E_{i+1})\rightarrow C^{\infty}_c(M,E_i)$ is the formal adjoint of $P$. Now, as in the previous comment, let $P'_i=x^{\nu}P$ be the $b-$operator that is naturally associated to $P$. It is well known that $\sigma_b^{\mu}(P'_{i+1}\circ P_i')=\sigma_b^{\mu}(P'_{i+1})\circ \sigma_b^{\mu}(P'_i)$ and that $\sigma_b^{\mu}((P'_i)^t)=(\sigma_b^{\mu}(P'_i))^t$. The proof follows now by standard arguments of linear algebra, in complete analogy with the case of an elliptic complex on a closed manifold.
\end{proof}
From the above proposition it follows the following useful corollary:
\begin{cor}
\label{icordelli}
In the same hypothesis of the previous proposition. The Hilbert complexes $(L^{2}(M,E_{*}),P_{max/min,*})$ are both Fredholm complexes. Moreover each Hilbert complex that extends $(L^{2}(M,E_{*}),P_{min,*})$ and that is extended by $(L^{2}(M,E_{*}),P_{max,*})$ is still an Fredholm complex.
\end{cor}
\begin{proof}
From Theorem \ref{gualdo} it follows that $P_{min,i}^t\circ P_{max,i}+P_{max,i-1}\circ P_{min,i-1}^t$ and $P_{max,i}^t\circ P_{min,i}+P_{min,i-1}\circ P_{max,i-1}^t$ are both Fredholm operators on their natural domain endowed with the graph norm. Now the statement follows  from  Prop. \ref{fonteavellana} 
\end{proof}

We remark the fact that we gave the definition of an elliptic complex of differential cone operators on a manifold with boundary $\overline{M}$. Following the remark after Definition \ref{iconti} the notion of elliptic complex of differential cone operators is naturally extended on a manifold $X$ with conical singularities.

\subsection{A brief reminder on the heat kernel}
The aim of this subsection is to recall briefly the main local  properties of the heat kernel on an open and oriented riemannian manifold $(M,g)$. \\Let $(M,g)$ be an open and oriented riemannian manifold, $E$ a vector bundle over $M$, $P_0:C_{c}^{\infty}(M,E)\rightarrow C^{\infty}_{c}(M,E)$ a non-negative symmetric differential operator and $P:\mathcal{D}(P)\subset L^{2}(M,E)\rightarrow L^{2}(M,E)$ a non-negative, self-adjoint extension of $P_0$. It is well know that, using the spectral theorem for unbounded self-adjoint operators and its associated functional calculus (see \cite{DUS}, chap. XXII), it is possible to construct the operator  $e^{-tP}$. The next result we are going to recall summarizes the main local properties of $e^{-tP}$ that we will use in the rest of the paper. We start with the following definitions:

\begin{defi}
A cut-off function is a smooth function $\eta:[0,\infty)\rightarrow [0,1]$ which admits  a $\epsilon>0$ such that $\eta(x)=1$ for $x\leq \frac{\epsilon}{4}$ and $\eta=0$ for $x\geq \epsilon$.
\end{defi}

\begin{defi}
\label{tranquillo}
Let $(M,g)$ be an open manifold, $E$ a vector bundle over $M$ and $P_0:C^{\infty}_{c}(M,E)\rightarrow C^{\infty}_{c}(M,E)$ a differential operator of second order. Then $P_0$ is a generalized Laplacian if   its principal symbol satisfies: $$\sigma^{2}(P_0)(x,\xi)=\|\xi\|^{2}.$$
\end{defi} 
An operator of this type is clearly elliptic. We refer to \cite{BGV} for a comprehensive  discussion on this class of operators.
\begin{teo}
\label{casale}
Let $(M,g)$ be an open and oriented riemannian manifold, $E$ a vector bundle over $M$, $P_0:C_{c}^{\infty}(M,E)\rightarrow C^{\infty}_{c}(M,E)$ a non-negative symmetric differential operator of order $d$ and $P:\mathcal{D}(P)\subset L^{2}(M,E)\rightarrow L^{2}(M,E)$ a non-negative, self-adjoint extension of $P$. Then $e^{-tP}$ satisfies the following properties:
\begin{itemize}
\item $e^{-tP}$ has a $C^{\infty}-$kernel, that is usually  labeled $e^{-tP}(s,q)$ or $k_{P}(t,s,q)$, which lies in $C^{\infty}((0,\infty)\times M\times M, E\boxtimes E^*).$
\item If $K_1,K_2$ are compact subset of $M$ such that $K_{1}\cap K_{2}=\emptyset$ then $$\|k_{P}(t,s,q)\|_{C^{k}(K_1\times K_{2},E\boxtimes E^*)}=O(t^n),\ t\rightarrow 0$$for all $k,n\in \mathbb{N}$.
\item
Let $\phi, \chi\in C^{\infty}_{c}(M)$; then the operator $\phi e^{-tP}\chi$ is a trace-class operator and we have, on   $C^l(K_1\times K_2,E\boxtimes E^*|_{K_1\times K_2})$ for each $l\in \mathbb{N}$, $$(\phi e^{-t P}\chi)(q,q)\sim_{t\rightarrow 0}\sum_{n=0}^{\infty}\phi(q)\chi(q)\Phi_{n}(q)t^{\frac{n-m}{d}}$$ and $$\Tr((\phi e^{-t P}\chi)(q,q))\sim_{t\rightarrow 0}\sum_{n=0}^{\infty}(\int_{M}\phi(q)\chi(q)\tr(\Phi(q))dvol_{g})t^{\frac{n-m}{d}}$$ where $q\in M$, $\{\Phi_{1},...,\Phi_{n},...,\}$ is a suitable sequence of sections in $C^{\infty}(M,End(E))$, $K_1=supp(\phi)$ and $K_2=supp(\chi)$.
\end{itemize}
Finally if $P_{0}$ is a generalized Laplacian then the last  property above can be sharpened  in the following way:
\begin{itemize}
\item 
Let $\phi, \chi\in C^{\infty}_{c}(M)$; then the operator $\phi e^{-tP}\chi$ is a trace-class operator and we have $$\phi(s)e^{-tP}(s,q)\chi(q)\sim_{t\rightarrow 0}h_{t}(s,q)\sum_{n=0}^{\infty}\phi(s)\chi(q)\Phi_{n}(s,q)t^{n}$$  where $(s,q)\in M\times M$, $\{\Phi_{1},...,\Phi_{n},...,\}$ is a suitable sequence of sections in $C^{\infty}(M\times M,E\boxtimes E^*)$ and $h_{t}(s,q)=(4\pi t)^{\frac{-n}{2}}e^\frac{-d(s,q)^2}{4t}\eta(d(s,q)^2)$ with $\eta$ a cut-off function. As in the previous case the above expansion holds in $C^l(K_1\times K_2,E\boxtimes E^*|_{K_1\times K_2})$ for each $l\in \mathbb{N}$, where $K_1=supp(\phi)$ and $K_2=supp(\chi)$.
\end{itemize}
\end{teo}

\begin{proof}
For the first three properties we refer to \cite{MAL}, Theorem 1.1.18. As explained there these properties are proved globally, for example in \cite{PBG}, when $M$ is a closed manifold. A careful examination of those proofs shows that the same properties remain true locally when $M$ is an open manifold. The same argumentation applies to the last property which is proved globally, on a closed manifold, in \cite{BGV} Prop. 2.46 or in \cite{JR} Theorem 7.15.
\end{proof}
The rest of the subsection is a brief reminder about the heat kernel of a differential cone operator. For more details and for the proof we refer to \cite{MAL}. As already recalled in Theorem \ref{gualdo} we know that, if $\overline{M}$ is a compact and oriented manifold with boundary, $M$ its interior part,  $P_0\in \dif_0(M,E;E)$ is a positive operator and $g$ is a conic metric over $M$, then for each positive self-adjoint extension $P$ of $P_0$, $e^{-tP}:L^2(M,g)\rightarrow L^2(M,g)$ is a trace-class operator.  Now we want to recall an important property named \textbf{scaling property}. Before doing this we need to introduce some notations:\\
Let $N$ be a compact manifold; consider $C(N)$ and endow it with a product metric $g=dr^2+h$ where $h$ is a riemannian metric over $N$. Finally let $E$ be a vector bundle over $reg(C(N))$.\\ Define $U_t:L^2(reg(C(N)),E)\rightarrow L^2(reg(C(N)),E)$ as $s(r,p)\mapsto t^{\frac{1}{2}}s(tr,p).$ It is immediate to show that $U_t:L^2(reg(C(N)),E)\rightarrow L^2(reg(C(N)),E)$ is an isometry and that $U_{t_1}\circ U_{t_2}=U_{t_1t_2}.$
\begin{prop}
\label{dodoma}
Let $N$ be a compact manifold, $E$ a vector bundle over $reg(C(N))$,  let $P_0\in \dif_0^{\mu,\nu}(reg(C(N)),E,E)$ be a symmetric differential cone operator and let $P$ be a self-adjoint extension of $P_0$. Endow $reg(C(N))$ with a product metric $g$, that is $g=dr^2+h$ where $h$ is a riemannian metric over $N$. Finally let $P_t=t^{\nu}U_tPU_t^*$ and let $f:\mathbb{R}\rightarrow \mathbb{R}$ a function such that $f(P)$ has a measurable kernel. Then for each $\lambda>0$
\begin{equation}
\label{atene}
f(P)(r, p, s, q) =\frac{1}{\lambda}f(\lambda^{-\nu}P_{\lambda})(\frac{r}{\lambda},p,\frac{s}{\lambda},q),\ \lambda>0
\end{equation}
As  particular case, given $P_0\in \dif_0^{\mu,\nu}(reg(C(N)),E,E)$ positive and $P$ a positive self-adjoint extension then 
\begin{equation}
\label{belgrado}
e^{-tP}(r,p,r,q)=\frac{1}{r}e^{-tr^{-\nu}P_r}(1,p,1,q)
\end{equation}
\end{prop}
\begin{proof}
See \cite{MAL} Lemma 2.2.3.
\end{proof}
Now we modify the above proposition for the heat operator in the case that $g$ is a conic metric over $M$.  As we will see, we are interested to the study of the $L^2-$Lefschetz numbers where the $L^2$ space are built using a conic metric. The reason is that when the considered  complex is the $L^2$ de Rham complex (built using a conic metric) then its $L^2-$cohomology has a topological meaning. More precisely, as showed by Cheeger in \cite{JEC}, we have the following theorem:
\begin{teo}
\label{ragusa}
Let $(F,h)$ be a compact and oriented riemannian manifold of dimension $f$. Consider the cone $C_b(F)$ with  $b$ a positive real number and endow $C_b(F)$ with the conic metric $g=dr^2+r^2h$. Then
\begin{equation}
\label{stia}
H^i_{2,max}(C_b(F),g)\cong
\left\{
\begin{array}{ll}
H^i(F) & i<\frac{f}{2}+\frac{1}{2}\\
0 & i\geq \frac{f}{2}+\frac{1}{2}
\end{array}
\right.
\end{equation}
If $X$ is a compact and oriented manifold with conical singularities and if $g$ is a conic metric over $reg(X)$ then 
\begin{equation}
\label{aosta}
H^i_{2,max}(reg(X),g)\cong I^{\underline{m}}H^i(X),\ H^i_{2,min}(reg(X),g)\cong I^{\overline{m}}H^i(X).
\end{equation}
\end{teo}
\begin{proof}
See \cite{JEC}.
\end{proof}
For the definition and the main  properties of intersection cohomology we refer to \cite{GM} and \cite{GMA}

\begin{lemma}
\label{ravenna}
Let $N$ be a compact manifold of dimension $n$, $E$ a vector bundle over $reg(C(N))$,  let $P_0\in \dif_0^{\mu,\nu}(reg(C(N)),E,E)$ be a positive differential cone operator and let $P$ be a positive self-adjoint extension of $P_0$. Endow $reg(C(N))$ with a conic metric $g$, that is $g=dr^2+r^2h$ where $h$ is a riemannian metric over $N$. Then for each $\lambda>0$
\begin{equation}
\label{varsavia}
e^{-tP}(r, p, s, q) =\frac{1}{\lambda^{n+1}}e^{-t\lambda^{-\nu}P_{\lambda}}(\frac{r}{\lambda},p,\frac{s}{\lambda},q),\ \lambda>0
\end{equation}
In particular we have 
\begin{equation}
\label{cracovia}
e^{-tP}(r, p, r, q) =\frac{1}{r^{n+1}}e^{-tr^{-\nu}P_{r}}(1,p,1,q).
\end{equation}
\end{lemma}
\begin{proof}
The proof is completely analogous to the proof of Proposition \ref{dodoma}. We have just to add the natural modifications caused by the fact that now the Hilbert space $L^2(reg(C(N)),E)$ is built using the conic metric $g=dr^2+r^2h$ and this means that given $\gamma\in L^2(reg(C(N)),E)$ we have $\|\gamma\|_{L^2(reg(C(N)),E)}=\int_{reg(C(N))}\langle \gamma, \gamma\rangle r^ndrdvol_h$ where $\langle \gamma, \gamma\rangle$ is the pointwise inner product induced by the metric on $E$ (which is a riemannian metric if E is a real vector bundle and is a Hermitian metric if $E$ is complex.). This implies that now the isometry $U_t$, introduced above Proposition \ref{dodoma}, is defined as $U_t:L^2(reg(C(N)),E)\rightarrow L^2(reg(C(N)),E),\ U_t(\gamma)=t^{\frac{n+1}{2}}\gamma(tr,p)$.  The proof follows now in completely analogy to that one of Proposition \ref{dodoma} . Moreover, in the case that $P$ is a positive self-adjoint extension of $\Delta_i:\Omega_c^i(reg(C(N)))\rightarrow \Omega_c^i(reg(C(N)))$, the Laplacian   constructed using a conic metric and acting on the space of smooth $i-$forms with compact support, the proof is given in \cite{JC}, pag. 582.
\end{proof}

Finally we conclude the section with the following proposition; before to state it we introduce some notations.  Given $\lambda\in \mathbb{R}$ we define 
$$p^+(\lambda):=|\lambda+\frac{1}{2}|\ \text{and}$$
\begin{equation}
\label{mediolanum}
 p^-(\lambda):=\left\{
\begin{array}{ll}
|\lambda-\frac{1}{2}| & |\lambda|\geq \frac{1}{2}\\
\lambda-\frac{1}{2} & |\lambda|< \frac{1}{2}
\end{array}
\right.
\end{equation}
Moreover we recall that $I_a(x)$ is the modified Bessel function of order $a$. For the definition see \cite{MAL} pag. 67.

\begin{prop}
\label{saronno}
Let $(N,h)$ be a compact and oriented riemannian  manifold of dimension $n$. Consider $C(N)$ and let $E$ be a vector bundle over $reg(C(N))$ endowed with a metric $\rho$ (hermitian if it is complex o riemannian if it is real). Suppose that $E$ admits an extension over all $[0,\infty)\times N$ that we  denote $\overline{E}$. Let $E_N=\overline{E}|_N$ and suppose that $(E,\rho)$ is isometric to $\pi^*(E_N,\rho|_N)$ where $\pi:(0,\infty)\times N\rightarrow N$ is the natural projection. Finally  let $P:C^{\infty}_c(E)\rightarrow C^{\infty}_{c}(E)$ be an elliptic differential cone operator of order one. Then:
\begin{enumerate}
\item On $L^2(reg(C_2(N)),E)$ built with the product metric $g_p=dr^2+h$, if $P$ satisfies $P=\frac{\partial}{\partial r}+\frac{1}{r}S$, where $S\in \dif^1(N,E_N)$ is elliptic, we have 
\begin{equation}
\label{munda}
e^{-tP_{max}^t\circ P_{min}}(r,p,s,q)=\sum_{\lambda \in \spe S}\frac{1}{2t}(rs)^{\frac{1}{2}}I_{p^+(\lambda)}(\frac{rs}{2t})e^{-\frac{r^2+s^2}{4t}}\Phi_{\lambda}(p,q)
\end{equation}
 and  $$e^{-tP_{min}\circ P^t_{max}}(r,p,s,q)=\sum_{\lambda \in \spe S}\frac{1}{2t}(rs)^{\frac{1}{2}}I_{p^-(\lambda)}(\frac{rs}{2t})e^{-\frac{r^2+s^2}{4t}}\Phi_{\lambda}(p,q)$$ where $\Phi_{\lambda}(p,q)$ is the smooth kernel of $\Phi_{\lambda}:L^2(N,E_N)\rightarrow V_{\lambda}$, the orthogonal projection on the eigenspace $V_{\lambda}.$
\item On $L^2(reg(C_2(N)),E)$ built with  the conic metric $g_c=dr^2+r^2h$, if $P$ satisfies \\$P=\frac{n}{2r}+\frac{\partial}{\partial r}+\frac{1}{r}S$, where $S\in \dif^1(N,E_N)$ is elliptic, we have 
\begin{equation}
\label{tapso}
e^{-tP_{max}^t\circ P_{min}}(r,p,s,q)=\sum_{\lambda \in \spe S}\frac{1}{2t}(rs)^{\frac{1-n}{2}}I_{p^+(\lambda)}(\frac{rs}{2t})e^{-\frac{r^2+s^2}{4t}}\Phi_{\lambda}(p,q)
\end{equation}
 and $$e^{-tP_{min}\circ P^t_{max}}(r,p,s,q)=\sum_{\lambda \in \spe S}\frac{1}{2t}(rs)^{\frac{1-n}{2}}I_{p^-(\lambda)}(\frac{rs}{2t})e^{-\frac{r^2+s^2}{4t}}\Phi_{\lambda}(p,q)$$where $\Phi_{\lambda}(p,q)$ is the smooth kernel of  $\Phi_{\lambda}:L^2(N,E_N)\rightarrow V_{\lambda}$ , the orthogonal projection on the eigenspace $V_{\lambda}.$
\end{enumerate}
\end{prop}

\begin{proof}
The first assertion is proved in \cite{MAL}, see Proposition 2.3.11 and pag. 68. The second statement follows using the following argument. Only for the remaining part  of this proof let us label $L^2(reg(C_2(N)),E, g_p)$ the $L^2$ space of sections built using the product metric $g_p=dr^2+h$ and  $L^2(reg(C_2(N)),E, g_c)$ the $L^2$ space of sections built using the conic metric $g_c=dr^2+r^2h$. The measure induced by $g_p$ is $drdvol_h$ while the measure induce by $g_c$ is $r^ndrdvol_h$. Therefore it is clear that the map $\tau:L^2(reg(C(N)),E, g_c)\rightarrow L^2(reg(C_2(N)),E, g_p) $, $\tau(\gamma)=r^{\frac{n}{2}}\gamma$ is an isometry with inverse given by $\tau^{-1}(\gamma)=r^{\frac{-n}{2}}\gamma$ . A simple calculation shows that $\tilde{P}:=\tau^{-1}\circ P\circ \tau$ satisfies $\tilde{P}=\frac{\partial}{\partial r}+\frac{1}{r}S$. Therefore $\tilde{P}_{max}^t\circ \tilde{P}_{min}=r^{\frac{n}{2}}P_{max}^t\circ P_{min}r^{\frac{-n}{2}}$ and this implies  that  $$e^{-t\tilde{P}_{max}^t\circ \tilde{P}_{min}}=r^{\frac{n}{2}}e^{-tP_{max}^t\circ P_{min}}r^{\frac{-n}{2}}.$$ Therefore if we call $\tilde{k}(t,r,p,s,q)$ the heat kernel relative to $e^{-t\tilde{P}_{max}^t\circ \tilde{P}_{min}}$ and analogously $k(t,r,p,s,q)$ the heat kernel relative to $e^{-tP_{max}^t\circ P_{min}}$ we have, for each $\gamma \in L^2(reg(C_2(N)),E, g_p)$, $$\int_{reg(C_2(N))}\tilde{k}(t,r,p,s,q)\gamma(s) dsdvol_h=\int_{reg(C_2(N))}r^{\frac{n}{2}}k(t,r,p,s,q)s^{\frac{-n}{2}}\gamma(s)s^ndsdvol_h$$ and therefore $\tilde{k}(t,r,p,s,q)=r^{\frac{n}{2}}k(t,r,p,s,q)s^{\frac{n}{2}}$. Finally, applying  this last equality to \eqref{munda},  we get  \eqref{tapso}. For the heat kernel of $e^{-tP_{min}\circ P_{max}^t}$ the proof is completely analogous to the previous one.
\end{proof}


\section{Geometric endomorphisms }
The goal of this section is to introduce and study the notion of \textbf{geometric endomorphism} of an elliptic complex of differential cone operators. 
\\Let $X$ be  a compact manifold with conical singularities and let $M$ be its regular part that, as explained after Definition \ref{iconti}, we identify with the interior part of $\overline{M}$ the manifold with boundary which desingularizes $X$, see Prop. \ref{palcano}. Finally consider an elliptic complex of differential cone operators as described in Definition \ref{perugia}:

\begin{equation}
\label{palermo}
0\rightarrow C^{\infty}_{c}(M,E_{0})\stackrel{P_{0}}{\rightarrow}C^{\infty}_{c}(M,E_{1})\stackrel{P_{1}}{\rightarrow}...\stackrel{P_{n-1}}{\rightarrow}C^{\infty}_{c}(M,E_n)\stackrel{P_{n}}{\rightarrow}0
\end{equation} 

\begin{defi}
\label{velletri}
A \textbf{geometric endomorphism} $T$ of \eqref{palermo} is given by a $n-$tuple of maps $T=(T_1,...,T_n)$, where each $T_i$ maps $C^{\infty}(M,E_i)$ to itself, constructed in the following way: there exists a smooth map $f:\overline{M}\rightarrow \overline{M}$ and a $n-$tuples of morphisms of bundles $\phi_i:f^*E_i\rightarrow E_i$ such that the following properties hold:
\begin{enumerate}
\item $f: \overline{M}\rightarrow \overline{M}$ is a diffeomorphism.
\item If, with a little abuse of notation, we still label with $f:X\rightarrow X$ the isomorphism  that $f:\overline{M}\rightarrow \overline{M}$ induces on $X$ then we require that $f(q)=q$ for each $q\in sing(X)$.
\item $T_i=\phi_i\circ f^*$ where $f^*$ acts naturally between $C^{\infty}(M,E_i)$ and $C^{\infty}(M,f^*E_i)$.
\item $P_i\circ T_i=T_{i+1}\circ P_i$.
\end{enumerate}
\end{defi}

We make a little comment on the above definition.  
The  third and the fourth property are exactly the definition of geometric endomorphism of an elliptic complex over a closed manifold given in \cite{ABL1}.  
However our definition is not a complete extension of that one given by Atiyah and Bott in \cite{ABL1}.  The reason is that in the closed case any smooth map is allowed. For our purposes we need that $T_i$ induce a bounded map from $L^2(M,E_i)$  to itself and clearly this prevents us to allow every smooth map in Definition \ref{velletri}. As we will see in the following lemma, the property that $f:\overline{M}\rightarrow \overline{M}$ is a diffeomorphism is a reasonable sufficient condition in order  to get a bounded extension of $T_i$ on $L^2(M,E_i).$ 

\begin{lemma}
\label{sorano}
In the same hypothesis of the above definition the endomorphism $T$ satisfies that the following properties:
\begin{enumerate}
\item For each $i$ and for each  $\psi\in C^{\infty}_c(M,E_i)$ we have $T_i(\psi)\in C^{\infty}_c(M,E_i).$
\item For each $i$ $T_i$ extends as a bounded operator from $L^2(M,E_i)$ to itself; with a small abuse of notation, we denote this again by $T_i$. 
\item Let $T_i^*:L^2(M,E_i)\rightarrow L^2(M,E_i)$ be the adjoint of $T_i$. Then   for each  $\psi\in C^{\infty}_c(M,E_i)$ we have $T_i^*(\psi)\in C^{\infty}_c(M,E_i).$
\end{enumerate}
\end{lemma}

\begin{proof}
The first two properties follow immediately by the fact that $f:\overline{M}\rightarrow \overline{M}$ is a diffeomorphism and that $\overline{M}$ is compact. For the third properties, we observe first of all that $T_i$ admits an adjoint because it is densely defined and that $T_i^*$ is bounded and defined over the whole $L^2(M,E_i)$ because $T_i$ is bounded. Now consider the bundle $f^*E_i$. The metric over $E_i$ induces in a natural way through $f$ a metric over $f^*E_i$. Therefore it make sense consider the bundle homomorphism $\phi_i^*:E_i\rightarrow f^*E_i$ defined in each fiber as the adjoint of $\phi_i$. Now consider the pull-back under $f$ of the volume form $dvol_g$. Then there exists a smooth function $\tau$ such that $\tau dvol_g=f^*dvol_g$ and $\tau>0$ if $f$ preserves the orientation of $M$, $\tau<0$ if $f$ reverses the orientation of $M$.  Finally define   $S:C^{\infty}_c(M,E_i)\rightarrow C^{\infty}_{c}(M,E_i)$ as 
\begin{equation}
\label{chieri}
 S_i(\psi):=\left\{
\begin{array}{ll}
  \tau(\phi_i^*\circ (f^{-1})^*)(\psi) & \text{if}\ f\  \text{preserves the orientation}\\
- \tau(\phi_i^*\circ (f^{-1})^*)(\psi) & \text{if}\  f\  \text{reserves the orientation}
\end{array}
\right.
\end{equation}
It is immediate to check that for each $\psi_1,\psi_2\in C^{\infty}_c(M,E_i)$ we have $$<T_i(\psi_1),\psi_2>_{L^2(M,E_i)}=<\psi_1,S_i(\psi_2)>_{L^2(M,E_i)}.$$Therefore, over $C^{\infty}_c(M,E_i)$ , $T_i^*$ coincides with $S$ and so from this  the third property follows immediately.
\end{proof}

Now we state  the following property :
\begin{prop}
\label{sanrocco}
Let $M$ be an  open and oriented riemannian manifold and let $g$ be an incomplete riemannian metric on $M$. Let $E_0,...,E_n$ be a sequence of vector bundles over $M$ and consider a complex of differential operators:
\begin{equation}
\label{livorno}
0\rightarrow C^{\infty}_{c}(M,E_{0})\stackrel{P_{0}}{\rightarrow}C^{\infty}_{c}(M,E_{1})\stackrel{P_{1}}{\rightarrow}...\stackrel{P_{n-1}}{\rightarrow}C^{\infty}_{c}(M,E_n)\stackrel{P_{n}}{\rightarrow}0
\end{equation} 
Let $T=(T_0,...,T_n)$ be an  endomorphism of  \eqref{livorno} that satisfies the second and the third properties  of Lemma \ref{sorano}.
Then we have the following properties:
\begin{enumerate}
\item For each $i=0,...,n$, for each $s\in \mathcal{D}(P_{min,i})$ we have $T_i(s)\in \mathcal{D}(P_{min,i})$ and $P_{min,i}\circ T_{i}=T_{i+1}\circ P_{min,i}$.
\item For each $i=0,...,n$, for each $s\in \mathcal{D}(P_{max,i})$ we have $T_i(s)\in \mathcal{D}(P_{max,i})$ and $P_{max,i}\circ T_{i}=T_{i+1}\circ P_{max,i}$.
\end{enumerate}
\end{prop}

\begin{proof}
Let $i\in \{0,...,n\}$ and let $s\in \mathcal{D}(P_{min,i})$. Then there exists a sequence $\{s_j\}_{j\in \mathbb{N}}$ such that $s_j\rightarrow s$ in $L^2(M,E_i)$ and $P_i(s_j)\rightarrow P_i(s)$ in $L^2(M,E_{i+1})$.  By the assumptions, we know that $\{T_i(s_j)\}_{j\in \mathbb{N}}$ is a sequence of smooth sections with compact support contained in $C^{\infty}_c(M,E_i)$ such that  $T_i(s_j)\rightarrow T_i(s)$ in $L^2(M,E_i)$ and $T_{i+1}(P_i(s_j))\rightarrow T_{i+1}(P_i(s))$ in $L^2(M,E_{i+1})$. But $T_{i+1}(P_i(s_j))=P_i(T_{i}(s_j))$. Therefore $P_i(T_{i}(s_j))$ converges in $L^2(M,E_{i+1})$ and this implies that $T_{i}(s)\in \mathcal{D}(P_{min,i})$ and that $P_{min,i}\circ T_{i}=T_{i+1}\circ P_{min,i}$.\\Now we give the proof of the second statement. From the first part of the proof it follows that, if we look at $T_{i+1}\circ P_{min,i}$, $P_{min,i}\circ T_i$ as unbounded operator with domain $\mathcal{D}(P_{min,i})$ then $T_{i+1}\circ P_{min,i}=P_{min,i}\circ T_i$ and therefore $(T_{i+1}\circ P_{min,i})^*=(P_{min,i}\circ T_i)^*$. Moreover, by the fact that $T_{i+1}$ is bounded, it follows that $(T_{i+1}\circ P_{min,i})^*=P_{min,i}^*\circ T_{i+1}^*$ with domain given by $(T_{i+1}^*)^{-1}(\mathcal{D}(P_{min,i}^*))$. Now let $s\in \mathcal{D}(P_{max,i})$ and let $\phi\in C^{\infty}_c(M,E_{i+1})$. Then $$<T_{i}(s),P_i^t(\phi)>_{L^2(M,E_i)}=<s,T_i^*(P_i^t(\phi))>_{L^2(M,E_i)}=<s,(P_{min,i}\circ T_i)^*(\phi)>_{L^2(M,E_i)}=$$  $$=<s,P_{min,i}^*(T_{i+1}^*(\phi))>_{L^2(M,E_i)}=   (\text{because}\ T_{i+1}^*(\phi)\in C^{\infty}_c(M,E_{i+1}))  $$ $$=<s,P_{max,i}^*(T_{i+1}^*(\phi))>_{L^2(M,E_i)}=<P_{max,i}(s),(T_{i+1}^*(\phi))>_{L^2(M,E_i)}$$ $$=<T_{i+1}(P_{max,i}(s)),\phi>_{L^2(M,E_i)}.$$ So we can conclude that $T_i(s)\in \mathcal{D}(P_{max,i})$ and that $T_{i+1}\circ P_{max,i}=P_{max,i}\circ T_{i}.$ 
\end{proof}

In the rest of this section we describe the notion of \textbf{non degeneracy condition} for a fixed point of a map $f:X\rightarrow X$. As we will see, over the regular part of $X$, this is the same of the one used in \cite{ABL1}. \\Let $X$ be a compact manifold with conical singularities and let  $f:X\rightarrow X$ a continuous map  such that $f(sing(X))\subset sing(X)$, $f(reg(X))\subset reg(X)$ and $f|_{reg(X)}$ is a smooth map. Define 
\begin{equation}
\label{buotano}
Fix(f):=\{p\in X: f(p)=p\}
\end{equation}
\begin{defi}
\label{ilpiano}
A point $p\in reg(X)\cap Fix(f)$ is said to be simple if $det(Id-d_{p}f)\neq 0$.
\end{defi}
Obviously this definition make sense because,  being $p$   a fixed point, it follows that $d_pf$ is an endomorphism of $T_p(reg(X))$. Moreover it is easy to show that Definition \ref{ilpiano} is equivalent to require that, on $reg(X)\times reg(X)$, $\mathcal{G}(f)$ meets transversely $\Delta_{reg(X)}$ on $(p,p)$, where $\mathcal{G}(f)$ is the graph of $f|_{reg(X)}$ and $\Delta_{reg(X)}$ is the diagonal of $reg(X)$ . In this way we get the following useful corollary:
\begin{cor}
\label{fossatodivico}
Each simple fixed point in $reg(X)\cap Fix(f)$ is an isolated fixed point.
\end{cor}
Now, following \cite{NSSS}, \cite{NS} but with little modifications, we recall what is a simple fixed point $p\in Fix(f)\cap sing(X)$.
As we said above, we assumed that $f(sing(X))\subset sing(X)$ and that $f(reg(X))\subset reg(X)$. Therefore if $q\in sing(X)\cap Fix(f)$ is a fixed conical point it follows that, on a neighborhood $U_q\cong C_2(L_q)$ of $q$, $f$ takes the form:
\begin{equation}
\label{cerrone}
f(r,p)=(rA(r,p),B(r,p))
\end{equation}
 We make the additional assumption  that  $A(r,p)$ and $B(r,p)$ are smooth up to zero, that is $$A(r,p):[0,2)\times L_q\rightarrow [0,2)$$ is smooth up to $0$ and analogously $$B(r,p):[0,2)\times L_q\rightarrow L_q$$ is smooth up to $0$. Moreover, by the fact that $f(sing(X))\subset sing(X)$ and that $f(reg(X))\subset reg(X)$ it follows that $A(r,p)\neq 0$ for $r>0$. Obviously if our starting point is a diffeomorphism $\overline{f}:\overline{M}\rightarrow \overline{M}$ as in Definition \ref{velletri}, then these requirements are automatically satisfied.

\begin{defi}
\label{serramaggio}
A point $q\in Fix(f)\cap sing(X)$ is a  \textbf{simple} fixed point if for each $p\in L_q$ at least one of the following conditions is satisfied:
\begin{enumerate}
\item   $A(0,p)\neq 1$.
\item  $B(0,p)\neq p$.
\end{enumerate}
\end{defi}
A natural question follows from Definition \ref{serramaggio}: what is the meaning of these requirements? The answer is that if $f$ satisfies one of the two requirements above then a sequence of fixed points converging to $q$ cannot exists and therefore $q$ is an isolated fixed point. We can show this last properties in the following way: suppose that $\{(r_j,p_j)\}$ is a sequence  of fixed point of $f$ contained in $U_q\cong C_2(L_q)$ such that $r_j\rightarrow 0$ when $j\rightarrow \infty$. Then $\{p_j\}$ is a sequence of point in $L_q$ which is compact and therefore there exists a subsequence, that with a little abuse of notations we still label $\{p_j\}$, such that $p_j$ converges to some $p\in L_q$.  By the assumptions, for each $j$, $(r_j,p_j)=(r_jA(r_j,p_j), B(r_j,p_j))$.  Therefore $1=\lim_{j\rightarrow \infty}A(r_j,p_j)=A(0,p)$, $B(r_j,p_j)=p_j$ for each $j$ and this implies that $f$ does not satisfies both the properties of Definition \ref{serramaggio}.

 So we can state the following useful corollary:
\begin{cor}
\label{lacasella}
Let $X$ be a compact manifold with conical singularities and let $f:X\rightarrow X$ a map such that $f(sing(X))\subset sing(X)$, $f(reg(X))\subset reg(X)$, $f|_{reg(X)}:reg(X)\rightarrow reg(X)$ is smooth and, on a neighborhood of a conical point, $A(r,p)$ and $B(r,p)$ are smooth up to 0. Then, if $f$ has only simple fixed point, $Fix(f)$ is made of a finite number of points.
\end{cor}
\begin{proof}
If  $f$ has only simple fixed points  then we already know that each of this fixed points is an isolated fixed point and this implies that $Fix(f)$ is a sequence without accumulation points. Therefore, by the compactness of $X$, it follows that $Fix(f)$ is made of a finite number of points.
\end{proof}

Now we state  the following definition:

\begin{defi}
\label{foggia}
Let $f$ be as in the previous corollary. Let $q\in Fix(f)\cap sing(X)$ a simple fixed point for $f$ such that $f$ satisfies the first requirement of Definition \ref{serramaggio}. Then if for each $p\in L_q$ 
\begin{equation}
\label{sucre}
A(0,p)<1
\end{equation}
 $q$ is called \textbf{attractive simple fixed point} while if 
\begin{equation}
\label{sassoferrato}
A(0,p)>1
\end{equation}
then $q$ is called \textbf{repulsive simple fixed point}.
\end{defi}

Clearly if for each $q\in sing(X)$ the relative link $L_q$ is connected then each simple fixed point $q\in sing(X)$ satisfying the first property of  Definition \ref{serramaggio}  is  necessarily attractive or repulsive. 

 Finally we conclude the section observing that in \cite{GMAC},  pag. 384, Goresky and MacPherson introduced the  notion of contracting fixed point. An elementary  check shows that \eqref{sucre} is equivalent to the definition given by Goresky and MacPherson.

\section{$L^2-$Lefschetz numbers of a geometric endomorphism}
Let $X$ be a compact manifold with conical singularities of dimension $m+1$. Consider an elliptic complex of cone differential operators as defined in Definition \ref{perugia}:
\begin{equation}
\label{catania}
0\rightarrow C^{\infty}_{c}(M,E_{0})\stackrel{P_{0}}{\rightarrow}C^{\infty}_{c}(M,E_{1})\stackrel{P_{1}}{\rightarrow}...\stackrel{P_{n-1}}{\rightarrow}C^{\infty}_{c}(M,E_n)\stackrel{P_{n}}{\rightarrow}0
\end{equation} 
where $P_i\in \dif_0^{\mu,\nu}(M,E_i,E_{i+1})$ and let $T=\phi\circ f$ be a geometric endomorphism of \eqref{catania} as in Definition \ref{velletri}. Obviously, with a small abuse of notation, we are using the same notation for  the diffeomorphism $f:\overline{M}\rightarrow \overline{M}$  and for the isomorphism that it induces on $X$. We recall that  the isomorphism $f:X\rightarrow X$ satisfies:
\begin{enumerate}
\item $f|_{reg(X)}:reg(X)\rightarrow reg(X)$ is a diffeomorphism
\item For each $p\in sing(X)$ we have $f(p)=p$
\item $A(r,p)$ and $B(r,p)$ (see \eqref{cerrone}) are smooth up to $0$.
\end{enumerate}  
Using Corollary \ref{icordelli} we know that both the complexes $(L^2(M,E_i),P_{max/min,i})$ are Fredholm complexes,  that is the cohomology groups $H^i_{2,max/min}(M,E_i)$ are finite dimensional.\\ Moreover by Proposition \ref{sanrocco} we know that $T$ is a morphism of both  complexes $(L^2(M,E_i),P_{max/min,i})$. Therefore, for each $i=0,...,n$, it induces an endomorphism $$T_i^*:H^i_{2,max}(M,E_i)\rightarrow H_{2,max}^i(M,E_i)\ \text{and\ analogously}\ T_i^*:H^i_{2,min}(M,E_i)\rightarrow H_{2,min}^i(M,E_i).$$
So we are in position to give the following definition:
\begin{defi}
\label{ferrara}
The $L^2-$Lefschetz numbers of $T$ are defined in the following way:
\begin{equation}
\label{amburgo}
L_{2,max}(T)=\sum_{i=0}^n(-1)^i\tr(T_i^*:H^i_{2,max}(M,E_i)\rightarrow H^i_{2,max}(M,E_i))
\end{equation}
and analogously
\begin{equation}
\label{berlino}
L_{2,min}(T)=\sum_{i=0}^n(-1)^i\tr(T_i^*:H^i_{2,min}(M,E_i)\rightarrow H^i_{2,min}(M,E_i))
\end{equation}
\end{defi}
The $L^2-$Lefschetz numbers satisfy the following property:
\begin{prop}
\label{ancona}
$L_{2,max/min}(T)$ do not depend on the conic metric $g$ we fix on $M$ and on the metrics $\rho_0,...,\rho_n$ that we fix on $E_0,...,E_n$
\end{prop}
\begin{proof}
By the fact that $\overline{M}$ is compact and that, as explained above Definition \ref{iconti}, $(E_i,\rho_i)$ are defined over all $\overline{M}$ and $\rho_i$ is non degenerate up to the boundary,  it follows that all the metrics  we  consider on $E_i$ are quasi-isometric. Moreover, using \cite{FB} Proposition 9, it follows that if $g$ and $g'$ are two conic metric over $M$ then they are quasi-isometric, that is there exists a positive real number $c$ such that  $g'\leq g\leq g'$. Therefore, for each $i=0,..,n$, $L^2(M,E_i)$ doesn't depend on the metric that we fix on $E_i$ and on the conic metric that we fix over $M$. This in turn implies that same conclusion holds for $H^i_{2,max}(M,E_i)$ and for $H^i_{2,min}(M,E_i)$, that is they  do not depend on the metric that we fix on $E_i$ and on the conic metric that we fix over $M$. In this way  we can conclude that also the traces of $T_i^*:H^i_{2,max}(M,E_*)\rightarrow H^i_{2,max}(M,E_*)$ and  $T_i^*:H^i_{2,min}(M,E_*)\rightarrow H^i_{2,min}(M,E_*)$  satisfy the same property and so the proposition is proved.
\end{proof}

\begin{itemize}
\item From the above proposition it follows that in order to calculate $L_{2,max/min}(T)$ we can use any conic metric $g$ on $M$ and any metrics $\rho_0,...,\rho_n$ over $E_0,...,E_n$. Therefore, in the remaining part of this section, we make the following assumptions: for each singular point $q$ there exists $U_q$, an open neighborhood of $q$ satisfying $U_q\cong C_2(L_q)$,  such that on  $reg(C_2(L_q))$    the conic metric $g$ satisfies $g=dr^2+r^2h$ where $h$  is any riemannian metric over $L_q$ that does not depend on $r$. Moreover we assume that  each metric $\rho_i$ on $E_i$ does not depend on $r$ in a collar neighborhood of $\partial \overline{M}$.
\end{itemize}

Consider, for each $i=0,...,n$, the operator $$\mathcal{P}_i:=P_i^t\circ P_i+P_{i-1}\circ P_{i-1}^t:C^{\infty}_c(M,E_i)\rightarrow C^{\infty}_c(M,E_i).$$ It is clearly a positive operator. As stated in Proposition \ref{lucca}, we know that $\mathcal{P}_i$ is an elliptic differential cone operator. Therefore, by Theorem \ref{gualdo}, we know that for each positive self-adjoint extension of $\mathcal{P}_i$, the relative heat operator is a trace-class operator. In particular this is true for  $\mathcal{P}_{abs,i}$ that we recall it is defined as  $P_{min,i}^t\circ P_{max,i}+P_{max,i-1}\circ P_{min,i-1}^t$ and for $\mathcal{P}_{rel,i}$ that it is defined as  $P_{max,i}^t\circ P_{min,i}+P_{min,i-1}\circ P_{max,i-1}^t$.
 \  A well known and basic result of  operators theory (see \cite{JR}, Prop. 8.8) says that, given an Hilbert space $H$, the space of trace-class operators is a two sided ideal of $\mathcal{B}(H)$, the space of bounded operators of $H$, and that the trace doesn't depend on the order of composition. In this way we know that for each $i=0,...,n$ $$T_i\circ e^{-t\mathcal{P}_{abs/rel, i}}:L^2(M,E_i)\rightarrow L^2(M,E_i)$$ are  trace-class operator and that $\Tr(T_i\circ e^{-t\mathcal{P}_{abs/rel, i}})= \Tr(e^{-t\mathcal{P}_{abs/rel, i}}\circ T_i$)\footnote{This is the reason because we need  to require that $f:\overline{M}\rightarrow \overline{M}$ is a diffeomorphism. In this way each $T_i:L^2(M,E_i)\rightarrow L^2(M,E_i)$ is bounded and so we can conclude that $T_i\circ e^{-t\mathcal{P}_{abs/rel, i}}$ is a trace-class operator }. Moreover it is clear that $T_i\circ e^{-t\mathcal{P}_{abs/rel, i}}$ are operators with smooth kernel given by 
\begin{equation}
\label{tebe}
\phi_i\circ k_{abs,i}(t,f(x),y)\ \text{for}\ T_i\circ e^{-t\mathcal{P}_{abs, i}}
\end{equation}
 and analogously 
\begin{equation}
\label{ankara}
\phi_i\circ k_{rel,i}(t,f(x),y)\ \text{for}\ T_i\circ e^{-t\mathcal{P}_{rel, i}}
\end{equation}
where $k_{abs/rel,i}(t,x,y)$ are respectively the smooth kernel of $e^{-tP_{abs/rel,i}}$. In both the expressions above $\phi_i$  acts on the $x$ variable of $k_{abs/rel,i}(t,f(x),y)$ because $k_{abs/rel,i}(t,f(x),y)$ is a section of $f^*E_i\boxtimes E_i^*$  and $\phi_i:f^*E_i\rightarrow E_i$ is a morphism of bundle. So the kernels $\phi_i\circ k_{abs/rel,i}(t,f(x),y)$ are well defined and they are  smooth sections of $E\boxtimes E^*.$ \\
Now we are in position to state the following theorem which is one of the main results of this section:
\begin{teo}
\label{santiago}
Consider an elliptic complex of differential cone operators as in \eqref{catania} and let $T$ be a geometric endomorphism as in Definition \ref{velletri}. Then for each $t$:
\begin{equation}
\label{calcutta}
L_{2,max}(T)=\sum_{i=0}^n(-1)^i\Tr(T_ie^{-t\mathcal{P}_{abs,i}})
\end{equation}
and analogously 
\begin{equation}
\label{castravetera}
L_{2,min}(T)=\sum_{i=0}^n(-1)^i\Tr(T_ie^{-t\mathcal{P}_{rel,i}})
\end{equation}
In particular, in both the equalities, the member on the right hand side  does not  depend on $t$.
\end{teo}
We need to state  some propositions in order to prove the above theorem. We give the proof only for the complex $(L^2(M,E_i),P_{max,i})$. The other one is completely analogous.
\begin{lemma}
\label{malecchie}
Consider an abstract Fredholm complex as in \eqref{mm} and let $T$ be an endomorphism of this complex, that is $T=(T_0,...,T_n)$, for each $i=0,...,n$ $T_i:H_i\rightarrow H_i$ is bounded and $D_i\circ T_i=T_{i+1}\circ D_i$ on $\mathcal{D}(D_i)$. Let $\pi_i:H_i\rightarrow \mathcal{H}_i(H_*,D_*)$  be the orthogonal projection induced by the Kodaira  decomposition of Proposition \ref{beibei}. Then for each $i=0,..,n$ we have $$\Tr(\pi_i\circ T_i:\mathcal{H}^i(H_*,D_*)\rightarrow \mathcal{H}^i(H_*,D_*))=\Tr(T_i^*:H^i(H_*,D_*)\rightarrow H^i(H_*,D_*))$$
\end{lemma}
\begin{proof}
Let $\gamma:\mathcal{H}^i(H_*,D_*)\rightarrow H^i(H_*,D_*)$ the isomorphism of \eqref{rabat}. Then it is clear that $T_i^*$, that is the endomorphism of $H^i(H_*,D_*)$ induced by $T_i$, satisfies $T_i^*=\gamma\circ \pi_i\circ T_i\circ \gamma^{-1}$. Now from this it follows immediately that $\Tr(\pi_i\circ T_i:\mathcal{H}^i(H_*,D_*)\rightarrow \mathcal{H}^i(H_*,D_*))=\Tr(T_i^*:H^i(H_*,D_*)\rightarrow H^i(H_*,D_*))$.
\end{proof}

\begin{lemma}
\label{manaus}
We have the following properties.
\begin{enumerate}
\item Let $E_{i}(\lambda)$ be the eigenspace relative to $\mathcal{P}_{abs,i}$ and the eigenvalue $\lambda$. Then $E_{i}(\lambda)$ is finite dimensional and made of  eigensections which are smooth in the interior. 
\item For each  $\lambda\neq 0$ consider the following complex:
\begin{equation}
\label{madras}
......\stackrel{P^{\lambda}_{max,i-1}}{\rightarrow}E_{i}(\lambda)\stackrel{P^{\lambda}_{max,i}}{\rightarrow}E_{i+1}(\lambda)\stackrel{P^{\lambda}_{max,i+1}}{\rightarrow}E_{i+2}(\lambda)\stackrel{P^{\lambda}_{max,i+2}}{\rightarrow}...
\end{equation}
\end{enumerate}
where $P^{\lambda}_{max,i}:=P_{max,i}|_{E_i(\lambda)}$.   Then it is an acyclic complex.
\end{lemma}
\begin{proof}
Consider the eigenspaces $E_{i}(\lambda)$. That is finite dimensional for each $\lambda\neq 0$ follows by the fact that $e^{-t\mathcal{P}_{abs,i}}$ is a trace-class operator while that it is finite dimensional for $\lambda=0$ follows by the fact that $\mathcal{P}_{abs,i}$ is  a Fredholm operator on its domain endowed with the graph norm. Moreover elliptic regularity tells us that $E_{i}(\lambda)$ is made of  eigensections which are smooth in the interior.  Finally, given $\lambda>0, $ consider 
\begin{equation}
\label{xanten}
......\stackrel{P^{\lambda}_{max,i-1}}{\rightarrow}E_{i}(\lambda)\stackrel{P^{\lambda}_{max,i}}{\rightarrow}E_{i+1}(\lambda)\stackrel{P^{\lambda}_{max,i+1}}{\rightarrow}E_{i+2}(\lambda)\stackrel{P^{\lambda}_{max,i+2}}{\rightarrow}...
\end{equation}
where $P^{\lambda}_{max,i}:=P_{max,i}|_{E_i(\lambda)}$. \\Let $s\in Ker(P_{max,i})$. Then $\mathcal{P}_{abs,i}(s)=\lambda s=P_{max,i-1}(P^t_{min}(s))$. Therefore $s\in ran(P_{max,i-1})$ and this implies that \eqref{xanten} is a long exact sequences, or in other words, it is an acyclic complex.
\end{proof}
Now we state the last result we need to prove Theorem \ref{santiago}. We take it from \cite{ABL1}.
\begin{lemma}
\label{madrid}
Consider a complex of finite dimensional vector space
\begin{equation}
\label{gerona}
0\rightarrow V_0\stackrel{f_{0}}{\rightarrow}...\stackrel{f_{i-1}}{\rightarrow}V_{i}\stackrel{f_{i}}{\rightarrow}V_{i+1}\stackrel{f_{,i+1}}{\rightarrow}V_{i+2}\stackrel{f_{i+2}}{\rightarrow}...\stackrel{f_{n-1}}{\rightarrow}V_n\stackrel{f_{n}}{\rightarrow}0.
\end{equation}
and for each $i$ let $G_i:V_i\rightarrow V_i$ an endomorphism such that $f_i\circ G_i=G_{i+1}\circ f_i$. Then $$\sum_{i=0}^n(-1)^i\Tr(G_i)=\sum_{i=0}^n(-1)^i\Tr(G_i^*)$$ where $G_{i}^*$ is the endomorphism of the $i-$th cohomology group of the complex \eqref{gerona} induced by $G_i$.
\end{lemma}
\begin{proof}
See \cite{ABL1}.
\end{proof}

\begin{proof} (of Theorem \ref{santiago}). As said above we give the proof only for \eqref{calcutta}. The proof for  \eqref{castravetera}  is completely analogous. Consider the heat operator $e^{-t\mathcal{P}_{abs,i}}:L^2(M,E_i)\rightarrow L^2(M,E_i)$. By the third point of Theorem \ref{gualdo} it follows  that there exists an Hilbert base of $L^2(M,E_i)$, $\{\phi_{j}\}_{j\in \mathbb{N}}$, made of smooth  eigensections of $\mathcal{P}_{abs,i}$, in such way the smooth kernel of $e^{-t\mathcal{P}_{abs,i}}$ satisfies $k(t,x,y)=\sum_{j}e^{-t\lambda_j}\phi_j(x)\boxtimes \phi_j^*(y)$. Moreover, by the fact that $T_i:L^2(M,E_i)\rightarrow L^2(M,E_i)$ is bounded, we know that $T_{i}\circ e^{-t\mathcal{P}_{abs,i}}$ and $e^{-t\mathcal{P}_{abs,i}}\circ T_i$ are trace class and that $\Tr(T_{i}\circ e^{-t\mathcal{P}_{abs,i}})=\Tr(e^{-t\mathcal{P}_{abs,i}}\circ T_i)$. Now, if we label $\pi(i,\lambda_j)$ the orthogonal projection $\pi(i,\lambda_j):L^2(M,E_i)\rightarrow E_{i}(\lambda_j)$, then we can write $e^{-t\mathcal{P}_{abs,i}}=\sum_{j}e^{-t\lambda_j}\pi(i,\lambda_j)$ and therefore $e^{-t\mathcal{P}_{abs,i}}\circ T_i=(\sum_{j}e^{-t\lambda_j}\pi(i,\lambda_j))\circ T_i=\sum_{j}e^{-t\lambda_j}(\pi(i,\lambda_j)\circ T_i)$. In this way we get  
\begin{equation}
\label{sassari}
\Tr(T_{i}\circ e^{-t\mathcal{P}_{abs,i}})=\Tr(e^{-t\mathcal{P}_{abs,i}}\circ T_i)=\sum_{j}e^{-t\lambda_j}\Tr((\pi(i,\lambda_j)\circ T_i)).
\end{equation}
Consider $\sum_{i=0}^n(-1)^i\Tr(T_i\circ e^{-t\mathcal{P}_{abs,i}})$. Then $\sum_{i=0}^n(-1)^i\Tr(T_i\circ e^{-t\mathcal{P}_{abs,i}})=$

 \begin{equation}
\label{nuoro}
=\sum_{i=0}^n(-1)^i\sum_je^{-t\lambda_j}\Tr((\pi(i,\lambda_j)\circ T_i))=\sum_je^{-t\lambda_j}\sum_{i=0}^n(-1)^i\Tr((\pi(i,\lambda_j)\circ T_i)).
\end{equation}
 Now examine carefully this last expression. Both $\pi(i,\lambda_j)\circ T_i:L^2(M,E_i)\rightarrow E_{i}(\lambda_j)$ and $\pi(i,\lambda_j):L^2(M,E_i)\rightarrow E_{i}(\lambda_j)$ are  trace-class operators. This implies that $\Tr(\pi(i,\lambda_j)\circ T_i)=\Tr(\pi(i,\lambda_j)\circ\pi(i,\lambda_j)\circ T_i)=\Tr(\pi(i,\lambda_j)\circ T_i\circ \pi(i,\lambda_j))$ and this last one is  equal to the trace of $\pi(i,\lambda_j)\circ T_i:E_{i}(\lambda_j)\rightarrow E_{i}(\lambda_j).$ But if  we take the following complex for $\lambda_j\neq 0$ 
\begin{equation}
\label{rovereto}
......\stackrel{P^{\lambda}_{max,i-1}}{\rightarrow}E_{i}(\lambda_j)\stackrel{P^{\lambda}_{max,i}}{\rightarrow}E_{i+1}(\lambda_j)\stackrel{P^{\lambda}_{max,i+1}}{\rightarrow}E_{i+2}(\lambda_j)\stackrel{P^{\lambda}_{max,i+2}}{\rightarrow}...
\end{equation}
we know that \eqref{rovereto} is an acyclic complex. Moreover it is immediate to check that  $\pi(i,\lambda_j)\circ T_i$ is an endomorphism of \eqref{rovereto} and therefore, applying Lemma \ref{gerona}, we can conclude that $\sum_{i=0}^n(-1)^i\Tr(\pi(i,\lambda_j)\circ T_i)=0$ for $\lambda_j\neq 0$. This leads to a relevant simplification of \eqref{nuoro}:
 \begin{equation}
\label{oristano}
\sum_{i=0}^n(-1)^i\Tr(T_ie^{-t\mathcal{P}_{abs,i}})=\sum_je^{-t\lambda_j}\sum_{i=0}^n(-1)^i\Tr(\pi(i,\lambda_j)\circ T_i)=\sum_{i=0}^n(-1)^i\Tr(\pi(i,0)\circ T_i).
\end{equation}
Finally, using Lemma \ref{malecchie}, it follows that $\Tr(\pi(i,0)\circ T_i)=\Tr(T_i^*)$ and therefore the theorem is proved.
\end{proof}

As an immediate consequence of Theorem \ref{santiago} we have the following corollary
\begin{cor}
\label{sandiego}
In the same assumptions of Theorem \ref{santiago} then
\begin{equation}
\label{lima}
L_{2,max}(T)=\lim_{t\rightarrow 0}\sum_{i=0}^n(-1)^i\Tr(T_i\circ e^{-t\mathcal{P}_{abs,i}})
\end{equation}
and analogously 
\begin{equation}
\label{brasilia}
L_{2,min}(T)=\lim_{t \rightarrow 0}\sum_{i=0}^n(-1)^i\Tr(T_i\circ e^{-t\mathcal{P}_{rel,i}})
\end{equation}
\end{cor}
 Before to go ahead we add some comments to Theorem \ref{santiago}.
\begin{rem}
In the statement of Theorem \ref{santiago} we assume that the endomorphism $T$  satisfies Definition \ref{velletri}. But from the proof it is clear that the particular structure of the endomorphism, that is $T_i=\phi_i\circ f^*$ doesn't play any role. It is just a sufficient condition to assure that each  $T_i$ induces a bounded map acting on $L^2(M,E_i)$ and that $T$ is an endomorphism of $(L^{2}(M,E_i),P_{max/min,i})$. Therefore if we have a $n-$ tuple of map $T=(T_1,...,T_n)$ such that, for each $i=0,...,n$,   $T_i:L^{2}(M,E_i)\rightarrow L^{2}(M,E_i)$      is bounded and $T_{i+1}\circ P_{max/min,i}=P_{max/min,i}\circ T_i$ on $\mathcal{D}(P_{max/min,i})$ then we can state and prove Theorem \ref{santiago} in the same way.
\end{rem}

\begin{rem}

\label{marsiglia}
We stated  Theorem \ref{santiago} in the case of an elliptic complex of differential cone operators over a compact manifold with conical singularities. This is because, using the result coming from the theory of elliptic differential cone operators, we know that $(L^{2}(M,E_i),P_{max/min,i})$ are Fredholm complexes and that $e^{-t\mathcal{P}_{abs/rel,i}}$ are trace-class operators. Therefore it is  possible to define maximal and minimal $L^2-$Lefschetz numbers and to prove Theorem \ref{santiago}. A priori it is not possible to do the same for an arbitrary elliptic complex of differential operators over a (possible incomplete) riemannian manifold $(M,g).$ But it is clear that  if we know that the maximal and the minimal extension of our complex are Fredholm complexes and that for each $i$ the heat operator constructed from the i-th laplacian associated to the maximal/minimal  complex is a trace-class operator, then it is possible to state and prove in the same way  formulas \eqref{calcutta} and \eqref{castravetera} for the $L^2-$Lefschetz numbers associated to the maximal and minimal extension of our complex.
\end{rem}

We conclude the section with the following theorems:
\begin{teo}
\label{urbana}
Let $X$ be a compact manifold with conical singularities of dimension $m+1$ and let $g$ be a conic metric on $reg(X)=M$. Consider an elliptic complex of differential cone operators as in \eqref{catania} and let $T=\phi\circ f^*$ be a geometric endomorphism of \eqref{catania} as in Definition \ref{velletri}.  Finally suppose that  $f$ has only simple fixed points. Then  we have: 
\begin{equation}
\label{benevento}
L_{2,max/min}(T)=\lim_{t\rightarrow 0}(\sum_{q\in Fix(f)}\sum_{i=0}^n(-1)^i\int_{U_q}\tr(T\circ e^{-t\mathcal{P}_{abs/rel,i}})dvol_g)
\end{equation}
 where $U_q$ is an open neighborhood of $q\in Fix(f)$ (clearly, when $q\in sing(X)\cap Fix(f)$ then we mean $U_q-\{q\}$).
\end{teo}
\begin{proof}
We know, by the assumptions, that  $f$ has only simple fixed points.  For each of these points, that we label $q$, let $U_q$ be an open neighborhood of $q$. Then, using again Corollary \ref{sandiego}, we know that $L_{2,max/min}(T)=\lim_{t\rightarrow 0}\sum_i(-1)^i\int_{M}\tr(T_i\circ e^{-t\mathcal{P}_{abs/rel,i}})$. Obviously we can break the member on the right as $$\sum_{q\in Fix(f)}\sum_{i=0}^n(-1)^i\int_{U_q}\tr(T_i\circ e^{-t\mathcal{P}_{abs/rel,i}})dvol_g+\sum_{i=0}^n(-1)^{i}\int_{V}\tr(T_i\circ e^{-t\mathcal{P}_{abs/rel,i}})dvol_g$$  where  $V=M-\cup_{q\in Fix(f)}U_q$. 
 Now, as remarked previously, we know that $f(q)=q$ for each $q\in sing(X)$. This implies $\{(f(q),q):\ q\in V\}$  is a compact subset of $M\times M$ disjoint from $\Delta_{M}$. So we can use the second property of Theorem \ref{casale} to conclude that $$\lim_{t\rightarrow 0}\int_{V}\tr(\phi_i\circ e^{-t\mathcal{P}_{abs/rel,i}}(f(q),q))dvol_g=\int_{V}\lim_{t\rightarrow 0}\tr(\phi_i\circ e^{-t\mathcal{P}_{abs/rel,i}}(f(q),q))dvol_g=0.$$ This complete the proof.
\end{proof}


The second point in the above theorem suggests to break the Lefschetz numbers as a contribution of two terms, that is \begin{equation}
\label{sapporo}
L_{2,max/min}(T)=\mathcal{L}_{max/min}(T,\mathcal{R})+\mathcal{L}_{max/min}(T,\mathcal{S})
\end{equation}
where $\mathcal{L}_{max/min}(T,\mathcal{R})$ is the contribution given by the simple fixed point lying in $reg(X)$, that is $$\mathcal{L}_{max/min}(T,\mathcal{R})=\lim_{t\rightarrow 0}(\sum_{q\in Fix(f)\cap reg(X)}\sum_{i=0}^n(-1)^i\int_{U_q}\tr(T_i\circ e^{-t\mathcal{P}_{abs/rel,i}})dvol_g)$$ and analogously $\mathcal{L}_{max/min}(T,\mathcal{S})$ is the contribution given by the simple fixed point lying in $Fix(f)\cap sing(X)$, that is $$\mathcal{L}_{max/min}(T,\mathcal{S})=\lim_{t\rightarrow 0}(\sum_{q\in Fix(f)\cap sing(X)}\sum_{i=0}^n(-1)^i\int_{U_q-\{q\}}\tr(T_i\circ e^{-t\mathcal{P}_{abs/rel,i}})dvol_g).$$

\begin{teo}
\label{sacramento}
In the hypothesis of the previous theorem, suppose furthermore that  for each $i=0,...,n$ $$P^t_i\circ P_i+P_{i-1}\circ P_{i-1}^t:C^{\infty}_c(M,E_i)\rightarrow C^{\infty}_{c}(M,E_i)$$ is a generalized Laplacian (see Definition \ref{tranquillo}). Then we get :
 $$L_{2,max}(T)=\sum_{q\in Fix(f)\cap M}\sum_{i=0}^n\frac{(-1)^i\Tr(\phi_i)}{|det(Id-d_qf)|}+\mathcal{L}_{2,max}(T,\mathcal{S}).$$ Analogously for $L_{2,min}(T)$ we have 
$$L_{2,min}(T)=\sum_{q\in Fix(f)\cap M}\sum_{i=0}^n\frac{(-1)^i\Tr(\phi_i)}{|det(Id-d_qf)|}+\mathcal{L}_{2,min}(T,\mathcal{S}).$$ 
\end{teo}
\begin{proof}
By Theorem \ref{urbana}, we know that the $L^2-$Lefschetz numbers depend only on the simple fixed point of $f$ and that we can localize their contribution, that is,  $$L_{2,max/min}(T)=\lim_{t\rightarrow 0}(\sum_{q\in Fix(f)}\sum_{i=0}^n(-1)^i\int_{U_q}\tr(T\circ e^{-t\mathcal{P}_{abs/rel,i}})dvol_g)$$ where $U_q$ is an arbitrary open neighborhood of $q$ (and clearly when $q\in sing(X)$ then we mean the regular part of $U_q$). If $q\in reg(X)\cap Fix(f)$, by the assumptions, we can use the local asymptotic expansion recalled in the last point of Theorem \ref{casale}. Now, to get the conclusion, the proof is exactly the same as in the closed case; see for example  \cite{BGV} Theorem 6.6 or \cite{JR} Theorem  10.12. 
\end{proof}
We have the following immediate corollary:
\begin{cor}
\label{adissabeba}
In the same hypothesis of Theorem \ref{sacramento}; 
Then: 
\begin{enumerate}
\item $\mathcal{L}_{max}(T,\mathcal{R})=\mathcal{L}_{min}(T,\mathcal{R})$ that is,  the simple fixed points in $M$ give the same contributions for both 
the Lefschetz numbers $L_{2,max/min}(T).$
\item $\mathcal{L}_{max/min}(T,\mathcal{S})$ do not depend on the particular conic metric  fixed on $M$ and  on the metrics $\rho_0,...,\rho_n$ respectively fixed on $E_0,...,E_n$.
\end{enumerate}
\end{cor}
\begin{proof}
The first assertion is an immediate consequence of the second point of Theorem \ref{sacramento}. For the second statement, by Proposition \ref{ancona}, we know that $L_{2,max/min}(T)$ are independent on the conic metric we put over $M$ and on the metric $\rho_0,...,\rho_n$ respectively on $E_0,...,E_n$. Again, by the second point of Theorem \ref{sacramento}, we know   that also $\mathcal{L}_{max/min}(T,\mathcal{R})$ are independent from the conic metrics and on  the metric $\rho_0,...,\rho_n$ respectively on $E_0,...,E_n$. Therefore the same conclusion holds for$\mathcal{L}_{max/min}(T,\mathcal{S})$.  The corollary is proved.
\end{proof}


\section{The contribution of the singular points}
The aim of this section is to give, in some particular cases, an explicit formula for $\mathcal{L}_{max/min}(T,\mathcal{S})$, that is for the contribution given by the singular points to the Lefschetz numbers $L_{2,max/min}(T).$\\ Consider the same situation described in  Theorem \ref{urbana}. Suppose moreover that the following properties hold:
\begin{enumerate}
\item For each $q\in sing(X)$ there exists an isomorphism $\chi_q:U_q\rightarrow C_2(L_q)$ such that on $[0,2)\times L_q$, using \eqref{genova}, each operator $A_k$ is \textbf{constant} in $x$ and, using the decomposition \eqref{cerrone}, the map $f$ takes the form: 
 \begin{equation}
\label{derby}
f=(rA(p),B(p)).
\end{equation}
\item  On $reg(C_2(L_q))$, using again the isomorphism $\chi_q:U_q\rightarrow C_2(L_q)$,    the conic metric $g$ satisfies $g=dr^2+r^2h$ with $h$  that does not depend on $r$ and  each metric $\rho_i$ on $E_i$ does not depend on $r$ in a collar neighborhood of $\partial \overline{M}$.
\end{enumerate} 

 Before  stating the next theorem we recall a definition from \cite{MAL}.
\begin{defi}
\label{teramo}
Consider the isometry $U_t:L^2(reg(C(N)),E)\rightarrow L^2(reg(C(N)),E)$ as defined in the proof of Lemma \eqref{ravenna}, that is $U_t(\gamma)=t^{\frac{n+1}{2}}\gamma(tr,p)$. Consider an operator $P_0\in \dif_0^{\mu,\nu}(reg(C(N)))$ such that, using the expression \eqref{genova}, each $A_k$ is constant in $x$. Then a closed extension $P$ of $P_0$ is said scalable if $U^*_tPU_t=t^{\nu}P$.
\end{defi}
\begin{lemma}
\label{jesi}
Given $P_0\in \dif_0^{\mu,\nu}(reg(C(N)))$ as in Definition \ref{teramo} then $P_{0,max}$ and $P_{0,min}$ are always scalable. If we take $P_0^t$, the formal adjoint of $P_0$, then also $P^t_{0,min}\circ P_{0,max}$, $P^t_{0,max}\circ P_{0,min}$, $P_{0,min}\circ P^t_{0,max}$ and  $P_{0,max}\circ P^t_{0,min}$ are scalable extensions of $P^t_0\circ P_0$ and $P_0\circ P_0^t$ respectively. Finally, if in a complex we consider $\mathcal{P}_i:=P_i^t\circ P_i+P_{i-1}\circ P_{i-1}^t$ (see the statement of Theorem \ref{sacramento}) then also the closed extension $\mathcal{P}_{abs,i}$ and $\mathcal{P}_{rel,i}$ (see \eqref{kaak} and \eqref{kkll}) are scalable extensions.
\end{lemma}
\begin{proof}
For the first assertion see \cite{MAL} pag. 58. The others assertions are an immediate consequence of the previous one and of the definition of scalable extension.
\end{proof}

Now we are ready to state  the following theorem:
\begin{teo}
\label{edimburgo}
In the same hypothesis of Theorem \ref{urbana}. Suppose moreover that the two properties  described above Definition \ref{teramo} hold. 
Then we have: 
\begin{equation}
\label{melbourne}
\mathcal{L}_{max/min}(T,\mathcal{S})=\sum_{q\in sing(X)}\sum_{i=0}^{n}(-1)^i\frac{1}{2\nu}\int_{0}^{\infty}\frac{dx}{x}\int_{L_q} \tr(\phi_i\circ e^{-x\mathcal{P}_{abs/rel,i}}(A(p),B(p),1,p))dvol_h.
\end{equation}
\end{teo}

\begin{proof}
Let $q\in sing(X)$. By the hypothesis we know that there exists an open neighborhood $U_q$ and an isomorphism $\chi_q:U_q\rightarrow C_2(L_q)$ such that, on $C_2(L_q)$, $f$ takes the form \eqref{derby} and each $A_k$ is constant in $x$.  Using  the properties stated in  \cite{MAL} pag. 42-43,  we get that the limit $$\lim_{t\rightarrow 0}\int_{reg(U_{q})}\tr(\phi_i\circ e^{-t\mathcal{P}_{abs/rel,i}}(rA(p),B(p),r,p))dvol_g$$ is equal to $$\lim_{t\rightarrow 0}\int_{reg(C_2(L_q))}\tr(\phi_i\circ e^{-t\mathcal{P}_{abs/rel,i}}(rA(p),B(p),r,p))r^mdvol_hdr$$ where, with a little abuse of notation, in the second expression we mean the heat kernel associated to  the absolute and relative extension of the operator, induced  by ${\mathcal{P}_{i}}|_{U_q}$ through $\chi_q$,  acting on $C^{\infty}_c(reg(C_2(L_q)),(\chi_q^{-1})^*E_i)$. So, for each $i=0,...,n$, we have to calculate  $$\lim_{t\rightarrow 0}\int_{reg(C_2(L_q))}\tr(\phi_i\circ e^{-t\mathcal{P}_{abs/rel,i}}(rA(p),B(p),r,p))r^mdrdvol_h.$$ Moreover, we assumed that, on $reg(C_2(L_q))$, the conic metric $g$ satisfies $g=dr^2+r^2h$ with $h$  that does not depend on $r$ and that each metric $\rho_i$ on $E_i$ does not depend on $r$ in a neighborhood of $\partial \overline{M}$. This implies that, for each $i=0,...,n$, the operator $\mathcal{P}_i$ satisfies the assumption at the beginning of the subsection, that is each $A_k$ does not depend on $x$. Therefore, using Lemma \ref{jesi}, we get that $\mathcal{P}_{abs/rel,i} $ are scalable extensions of $\mathcal{P}_i$. Now, after these observations, we can go on to calculate $$\lim_{t\rightarrow 0}\int_{reg(C_2(L_q))}\tr(\phi_i\circ e^{-t\mathcal{P}_{abs/rel}}(rA(p),B(p),r,p))dvol_g.$$ Using Lemma \ref{ravenna} and the fact that $\mathcal{P}_{abs/rel,i}$ are scalable extensions of $\mathcal{P}_i$ we get $$\int_{reg(C_2(L_q))}\tr(\phi_i\circ e^{-t\mathcal{P}_{abs/rel,i}}(rA(p),B(p),r,p))r^mdrdvol_h =$$  $$=\int_0^2\int_{L_q}\frac{1}{r}\tr(\phi_i\circ e^{-tr^{-2\nu}\mathcal{P}_{abs/rel,i}}(A(p),B(p),1,p))dvol_hhdr.$$ Now if we put $\frac{t}{r^{2\nu}}=x$ we get $\frac{-2\nu tdr}{r^{2\nu+1}}=dx$ which implies that  $$\frac{dx}{x}=\frac{-2\nu tdr}{r^{2\nu+1}}\frac{r^{2\nu}}{t}=-2\nu\frac{dr}{r}$$ Moreover when $r$ goes to $0$ then $x$ goes to $\infty$ and when $r$ goes to $2$ then $x$ goes to $\frac{t}{4}$. So we get $$\int_0^2\int_{L_q}\frac{1}{r}\tr(\phi_i\circ e^{-tr^{-2\nu}\mathcal{P}_{abs/rel,i}}(A(p),B(p),1,p))dvol_hhdr=$$ 
\begin{equation}
\label{massafra}
=\frac{1}{2\nu}\int_{t/4}^{\infty}\frac{dx}{x}\int_{L_q} \tr(\phi_i\circ e^{-x\mathcal{P}_{abs/rel,i}}(A(p),B(p),1,p))dvol_h.
\end{equation}
Therefore to conclude we have to evaluate the limit
\begin{equation}
\label{imperia}
\lim_{t\rightarrow 0}\frac{1}{2\nu}\int_{t/4}^{\infty}\frac{dx}{x}\int_{L_q} \tr(\phi_i\circ e^{-x\mathcal{P}_{abs/rel,i}}(A(p),B(p),1,p))dvol_h
\end{equation}
To do this consider the term $\int_{L_q} \tr(\phi_i\circ e^{-x\mathcal{P}_{abs/rel,i}}(A(p),B(p),1,p))dvol_h$. We know, by the hypothesis, that $f$ has only simple fixed points. In particular each $q\in  sing(X)$ is a simple fixed point. The conditions described in Definition \ref{serramaggio} together with \eqref{derby} implies that either $A(p)\neq 1$ for all $p\in L_q$ or $B:L_q\rightarrow L_q$ has not fixed points. Anyway each of these conditions implies that when $p$ runs over $L_q$ then $\{(A(p),B(p),1,p)\}$ is a compact subset of $reg(C_2(L_q))\times reg(C_2(L_q))$ that doesn't intersect the diagonal. Therefore we can use the second property stated in Theorem \ref{casale} to conclude that, when $x\rightarrow 0$,
\begin{equation}
\label{savona}
\int_{L_q} \tr(\phi_i\circ e^{-x\mathcal{P}_{abs/rel,i}}(A(p),B(p),1,p))dvol_h=O(x^N)\ \text{for each}\ N>0.
\end{equation}
In this way we can conclude that the limit \eqref{imperia} exists and we have

$$\lim_{t\rightarrow 0}\frac{1}{2\nu}\int_{t/4}^{\infty}\frac{dx}{x}\int_{L_q} \tr(\phi_i\circ e^{-x\mathcal{P}_{abs/rel,i}}(A(p),B(p),1,p))dvol_h=$$
\begin{equation}
\label{grosseto}
=\frac{1}{2\nu}\int_{0}^{\infty}\frac{dx}{x}\int_{L_q} \tr(\phi_i\circ e^{-x\mathcal{P}_{abs/rel,i}}(A(p),B(p),1,p))dvol_h.
\end{equation}
\end{proof}

Now, for each $i=0,..,n$, using again the hypothesis  and the notations of 
Theorem \ref{edimburgo}, and assuming still that $q$ is a simple fixed point for $f$, define the following "modified version" of the classical $\zeta-$function:
\begin{equation}
\label{rovigo}
\zeta_{T_i,q}(\mathcal{P}_{abs/rel,i})(s):=\frac{1}{2\nu}\int_{0}^{\infty}x^{s-1}dx\int_{L_q} \tr(\phi_i\circ e^{-x\mathcal{P}_{abs/rel,i}}(A(p),B(p),1,p))dvol_h.
\end{equation}
The definition makes sense for each $s\in \mathbb{C}$ because, as observed in the proof of Theorem \ref{edimburgo}, $\{(A(p),B(p),1,p)\}$ is a compact subset of $reg(X)\times reg(X)$ that is disjoint from the diagonal $\Delta_{reg(X)}$. Therefore we can apply the second point of Theorem \ref{casale} to conclude that, when $x\rightarrow 0$,
\begin{equation}
\label{bordighera}
\int_{L_q} \tr(\phi_i\circ e^{-x\mathcal{P}_{abs/rel,i}}(A(p),B(p),1,p))dvol_h=O(x^N)\ \text{for each}\ N>0.
\end{equation}
and this implies that $\zeta_{T_i,q}(\mathcal{P}_{abs/rel,i})(s)$ is a holomorphic function over the whole complex plane.
The reason behind \eqref{grosseto} is that if we compare  \eqref{grosseto} with the definitions of the zeta functions for a generalized Laplacian, see for example \cite{BGV} pag. 295, then it natural to think at \eqref{grosseto} as a sort of zeta function for the operators $\mathcal{P}_{abs/rel,i}$ valued in $0$, which takes account of the action of $T_i$ in its definition.
In this way, using \eqref{rovigo}, we can reformulate Theorem \ref{edimburgo} in a more concise  way:
\begin{equation}
\label{domodossola}
\mathcal{L}_{max/min}(T,\mathcal{S})=\sum_{q\in  sing(X)}\sum_{i=0}^{n}(-1)^i\zeta_{T_i,q}(\mathcal{P}_{abs/rel,i})(0).
\end{equation}

Before to conclude the section we make the following remarks.\\In the same hypothesis of  Theorem \ref{urbana} consider a point $q\in sing(X)$ such that $q$ is an attractive simple fixed point.  We recall that over a neighborhood  $U_q\cong [0,2)\times L_q$ of $q$ we can look at $f$ as a map given by $(rA(r,p),B(r,p)):[0,2)\times L_q\rightarrow [0,2)\times L_q$ with $A$ and $B$ smooth up to $0$. From Definition \ref{foggia} we know that $q$ is attractive if $A(0,p)<1$ for each fixed $p\in L_q$. Clearly this implies  that $f(U_q)\subset U_q$.
Therefore it follows that, if we consider the complex
\begin{equation}
\label{messina}
0\rightarrow C^{\infty}_{c}(U_q,E_{0}|_{U_q})\stackrel{P_{0}}{\rightarrow}C^{\infty}_{c}(U_q,E_{1}|_{U_q})\stackrel{P_{1}}{\rightarrow}...\stackrel{P_{n-1}}{\rightarrow}C^{\infty}_{c}(U_q,E_{n}|_{U_q})\stackrel{P_{n}}{\rightarrow}0
\end{equation} 
then $T$ is also a geometric endomorphism of \eqref{messina} and, using Proposition \ref{sanrocco}, we get that $T$ extends as a bounded endomorphism of the complexes $(L^2(U_q,E_{i}|_{U_q}),({P|_{U_q}})_{max/min,i})$.\\ Moreover, by the results proved in the first and the second chapter of  \cite{MAL}, it follows that $(L^2(U_q,E_{i}|_{U_q}),({P|_{U_q}})_{max/min,i})$ are both Fredholm complexes and that,  the respective heat operators , $e^{-t(\mathcal{P}|_{U_q})_{abs/rel,i}}:L^2(U_q,E_{i}|_{U_q})\rightarrow L^2(U_q,E_{i}|_{U_q})$,  are trace-class operators.\\Using again  the properties stated in \cite{MAL} at pag. 42-43, it follows that for each open neighborhood $V_q$ of $q$, such that  $\overline{V_q}$ is a  subset of $U_q$, we have  $$\lim_{t\rightarrow 0}\int_{V_q}\tr(\phi_i\circ e^{-t\mathcal{P}_{abs/rel,i}}(rA(r,p),B(r,p),r,p)dvol_g=$$ $$=\lim_{t\rightarrow 0}\int_{V_q}\tr(\phi_i\circ e^{-t(\mathcal{P}|_{U_q-\{q\}})_{abs/rel,i}}(rA(r,p),B(r,p),r,p)dvol_g.$$  
Suppose now that we are in the hypothesis of Theorem \ref{edimburgo}.
By the proof of the same theorem, it follows that for each $0<b\leq 2$ $$\lim_{t\rightarrow 0}\int_{0}^b \int_{L_q}\tr(\phi_i\circ e^{-t(\mathcal{P}|_{U_q-\{q\}})_{abs/rel,i}})(rA(p),B(p),r,p)r^mdvol_hdr=$$   $$\int_{0}^\infty x^{-1}dx\int_{L_q}\tr(\phi_i\circ e^{-x(\mathcal{P}|_{U_q-\{q\}})_{abs/rel,i}})(A(p),B(p),1,p)dvol_h$$
that is it does not depend on the particular $b$ we fixed. Therefore  we can conclude that
\begin{equation}
\label{semonte}
\lim_{t\rightarrow 0}\int_{U_q-\{q\}}\tr(\phi_i\circ e^{-t\mathcal{P}_{abs/rel,i}}(rA(p),B(p),r,p)dvol_g=
\end{equation}
$$=\lim_{t\rightarrow 0}\int_{U_q-\{q\}}\tr(\phi_i\circ e^{-t(\mathcal{P}|_{U_q-\{q\}})_{abs/rel,i}}(rA(p),B(p),r,p)dvol_g.$$
 Summarizing we obtained that  it makes sense  to define, for an attractive simple fixed point,  $L_{2,max/min}(T|_{U_q})$    as the $L^2-$Lefschetz numbers of $T$  acting on the maximal/minimal extension of \eqref{messina} and that, under the hypothesis of Theorem \ref{edimburgo}, it satisfies 
\begin{equation}
\label{kigali}
L_{2,max/min}(T|_{U_q})=\lim_{t\rightarrow 0}\sum_{i=0}^n(-1)^i\int_{U_q-\{q\}}\tr(\phi_i\circ e^{-t\mathcal{P}_{abs/rel,i}}(rA(p),B(p),r,p)dvol_g.
\end{equation}
Now we proceed making another remark before the conclusion.\\
As showed in the second section, $T_i^*$, the adjoint of $T_i$, has the following form: 
\begin{equation}
\label{benoni}
T_i^*=\theta_i\circ (f^{-1})^*
\end{equation} 
where $\theta_i=\tau\phi_i^*$ with $\tau$ positive or negative function respectively if $f$ preserves or reverses the orientation. Moreover, a simple computation, shows that $T^*$ is an endomorphism of the following Fredholm complexes: $(L^2(M,E_i),P^t_{max/min,i})$. By the fact that, if $Q:H\rightarrow H$ is a trace-class operator acting on the Hilbert space $H$ then also $Q^*$ is trace-class and $\Tr(Q)=\Tr(Q^*)$, it follows that 
\begin{equation}
\label{panicale}
\Tr(T_i\circ e^{-t\mathcal{P}_{abs/rel,i}})=\Tr( e^{-t\mathcal{P}_{abs/rel,i}}\circ T_i^*)= \Tr(T^*_i\circ  e^{-t\mathcal{P}_{abs/rel,i}}).
\end{equation} 
In particular, from \eqref{panicale}, it follows  that:
\begin{equation}
\label{talamello}
L_{2,max/min}(T)=L_{2,min/max}(T^*)
\end{equation}
where $T$ acts on $(L^2(M,E_i),P_{max/min,i})$ and $T^*$ acts on $(L^2(M,E_i),P^t_{min/max,i})$.\\A second consequence is the following: consider a point $q\in sing(X)$ such that $q$ is a repulsive simple fixed point. Clearly, by the fact that $f$ on $U_q\cong C_2(L_q)$ takes the form $f=(rA(p),B(p))$ it follows that $f^{-1}=(rG(p), B^{-1}(p))$ where $G=\frac{1}{A\circ B^{-1}}$. The fact that $q$ is repulsive means that $A>1$. Therefore it follows that $q$ is an \textbf{attractive} simple fixed point  for $T^*$.\\Finally we are in position to conclude with the following results: 

%

\begin{cor}
\label{vasto}
In the same hypothesis of Theorem \ref{edimburgo};
 Suppose moreover that  $q\in  sing(X)$ is an attractive fixed point. Then $$\sum_{i=0}^{n}(-1)^i\zeta_{T_i,q}(\mathcal{P}_{abs/rel,i})(0)=L_{2,max/min}(T|_{U_q}).$$ In particular this tells us that $\sum_{i=0}^{n}(-1)^i\zeta_{T_i,q}(\mathcal{P}_{abs/rel,i})(0)$ has a geometric meaning itself.
\end{cor}

\begin{proof}
It follows immediately from Theorem \ref{edimburgo} and \eqref{kigali}.
\end{proof}

\begin{teo}
\label{pietrogrado}
In the same hypothesis of Theorem \ref{sacramento}. Suppose moreover that the first property stated at the beginning of the section holds.  Then we have: 
\begin{equation}
\label{casablanca}
L_{2,max/min}(T)=\sum_{p\in Fix(f)\cap M}\sum_{i=0}^n\frac{(-1)^i\Tr(\phi_i)}{|det(Id-d_qf)|}+\sum_{q\in  sing(X)}\sum_{i=0}^n(-1)^i\zeta_{T_i,q}(\mathcal{P}_{abs/rel,i})(0)
\end{equation}
where in \eqref{casablanca} the contribution given by the singular points is calculated fixing any conic metric $g$ on $reg(X)$ and any metrics $\rho_0,...,\rho_n$ on $E_0,...,E_n$ which satisfy the hypothesis of Theorem \ref{edimburgo}.\\
Moreover if each point $q\in sing(X)$ is an attractive fixed point we have:
\begin{equation}
\label{sierranevada}
L_{2,max/min}(T)=\sum_{p\in Fix(f)\cap M}\sum_{i=0}^n\frac{(-1)^i\Tr(\phi_i)}{|det(Id-d_qf)|}+\sum_{q\in  sing(X)} L_{2,max/min}(T|_{U_q}).
\end{equation}

while if each $q\in sing(X)$ is a   repulsive fixed point then we have :
\begin{equation}
\label{fondarca}
L_{2,max/min}(T)=\sum_{p\in Fix(f)\cap M}\sum_{i=0}^n\frac{(-1)^i\Tr(\theta_i)}{|det(Id-d_q(f^{-1}))|}+\sum_{q\in  sing(X)} L_{2,min/max}(T^*|_{U_q}).
\end{equation}
Finally we remark again that, when $\mathcal{P}_i$ is a generalized Laplacian,  the contribution given by the singular simplex fixed points, that is  $$\mathcal{L}_{max/min}(T,\mathcal{S})=\sum_{q\in  sing(X)}\sum_{i=0}^n(-1)^i\zeta_{T_i,q}(\mathcal{P}_{abs/rel,i})(0)$$ does not depend on the particular  conic metric that we fix on $reg(X)$ and on the metrics $\rho_0,...,\rho_n$ that we fix on $E_0,...,E_n$.
\end{teo}

\begin{proof}
As showed in Corollary \ref{adissabeba}, when each $\mathcal{P}_i$ is a generalized Laplacian, then $L_{2,max/min}(T)$, $\mathcal{L}(T,\mathcal{R})$ and $\mathcal{L}_{max/min}(T,\mathcal{S})$ do not depend on the conic metric we fix on $reg(X)$ and do not depend on the metrics we fix $\rho_0,...,\rho_n$ on $E_0,...,E_n$. Therefore, without loss of generality, we can assume that for each $q\in sing(X)$, using the isomorphism $\chi_q:U_q\rightarrow C_2(L_q)$ of \eqref{derby},  the conic metric $g$ satisfies $g=dr^2+r^2h$ with $h$  that does not depend on $r$ and that each metric $\rho_i$ on $E_i$ does not depend on $r$ in a neighborhood of $\partial \overline{M}$. In this way we are in position to apply Theorem \ref{edimburgo} and so \eqref{casablanca} follows combining the  theorems \ref{sacramento}  and \ref{edimburgo}. Moreover this tell us that, in \eqref{casablanca},  the contribution of the singular points  is well defined and does not depend on the metrics $g,\rho_0,...,\rho_n$ (satisfying the assumptions of  Theorem \ref{edimburgo}) used to calculate it.  
 The second assertion follows from Corollary \ref{vasto} while the last assertion follows from \eqref{benoni} and \eqref{talamello}.
\end{proof}

\begin{rem}
We stress on the fact that, unlike Theorem \ref{edimburgo}, in Theorem \ref{pietrogrado} there are not assumptions about the conic metric $g$ on $reg(X)$  and about  the metrics $\rho_0,...,\rho_n$ on $E_0,...,E_n$ respectively.
\end{rem}

Finally we conclude the section with the following comment.\\ The condition that we required at the beginning of the subsection for each operator $P_i$, that each $A_k$ does not depend on $x$, might appear as to  be too strong at  first right. Obviously this is indeed a  strong assumption but it is at the same time  quite natural because the most natural complex arising in differential geometry, the de Rham complex, satisfies this assumption.\\The requirement \eqref{derby}, about the behavior of $f$ near the point $p$, is justified by the idea to evaluate $\mathcal{L}_{max/min}(T,\mathcal{S})$ using the scaling invariance of the heat kernel, see Lemma \ref{ravenna}. In fact if $f=(rA(r,p),B(r,p))$ then, after  the scaling invariance is used, we get in our expression the term $\tr(\phi_i\circ e^{-tr^{-2\nu}\mathcal{P}_{abs/rel,i}}(A(r,p),B(r,p),1,p))$. To have that this last expression make sense we need that  $(A(r,p),B(r,p),1,p)\in \mathcal{G}(f)$,  where $\mathcal{G}(f)\subset X\times X$ is the graph of $f$, and therefore this leads us to assume \eqref{derby}.

\subsection{The case of a short complex}
The aim of this subsection is to give a formula for the $L^2-$Lefschetz numbers in the particular case of a short complex, that is is an elliptic conic operator $P:C^{\infty}_{c}(M,E)\rightarrow C^{\infty}_{c}(M,E)$, using the result stated in Proposition \ref{saronno}. 
To do this we start describing  our geometric situation which is the same of  the previous results with some additional requirements: let $X$ be a compact and oriented manifold with conical singularities of dimension $m+1$. Let $M$ be its regular part and let $\overline{M}$ be the compact manifold with boundary which desingularize $X$. Endow $M$ with a conic metric $g$.
Let $(E,\rho)$ be a vector bundle endowed with a metric (riemannian or hermitian)  according if $E$ is complex or real. Let $(\overline{E},\rho)$ be the extension of $(E,\rho)$ over $\overline{M}$.  Let $T=(T_1,T_2)$  be a geometric endomorphism  where, as we already know, $T_i=\phi_i\circ f^*$ with $f:\overline{M}\rightarrow \overline{M}$ is a diffeomorphism as described in Definition \ref{velletri} and $\phi:f^*E\rightarrow E$ a bundle homorphism. Suppose that $Fix(f)$ is made only by simple fixed points. Finally, suppose that in each neighborhood $U_q\cong C_2(L_q)$ of $q\in sing(X)$  the operator $P$ take the form 
\begin{equation}
\label{munster}
P=\frac{n}{2r}+\frac{\partial}{\partial r}+\frac{1}{r}S
\end{equation}
 where $S\in \dif^1(N,E_N)$
 is an elliptic operator and the map $f$ take the form
\begin{equation}
\label{bergamo}
f=(rc,B(p)),\ c\neq 1
\end{equation}
where $c>0$ and depends only on $q$.

\begin{teo}
\label{granada}
In the same hypothesis of Theorem \ref{edimburgo}; suppose moreover that the properties described above hold. Then for each $q\in sing(X)$ we have: 
\begin{equation}
\label{lajolla}
\zeta_{T_0,q}(P_{max}^t\circ P_{min})(0)=\frac{c^{\frac{1-n}{2}}}{4}\int_0^{\infty}e^{-\frac{u(c^2+1)}{4}}\sum_{\lambda\in \spe S}I_{p^+(\lambda)}(\frac{uc}{2})du\Tr(\tilde{\Phi}_{0,\lambda,q})
\end{equation}
and analogously
\begin{equation}
\label{sandonaci}
\zeta_{T_1,q}(P_{min}\circ P_{max}^t)(0)=\frac{c^{\frac{1-n}{2}}}{4}\int_0^{\infty}e^{-\frac{u(c^2+1)}{4}}\sum_{\lambda\in \spe S}I_{p^-(\lambda)}(\frac{uc}{2})du\Tr(\tilde{\Phi}_{1,\lambda,q})
\end{equation}
where $$\Tr(\tilde{\Phi}_{j,\lambda,q})=\int_{L_q}\tr(\phi_j\Phi_{\lambda,q}(B(p),p))dvol_h,\   j=0,1.$$
\end{teo}

\begin{proof}
We give the proof only for \eqref{lajolla} because for \eqref{sandonaci} is completely analogous.
To prove the assertion we have to calculate $$\lim_{t\rightarrow 0}\int_{reg(C_2(L_q))}\tr(T_0\circ e^{-P_{max}^t\circ P_{min}})dvol_g.$$  By the assumptions we are in position to use  the second statement of Proposition \ref{saronno} and  therefore it is clear that the smooth kernel of $T_0\circ e^{-P_{max}^t\circ P_{min}}$ is 
\begin{equation}
\label{medina}
\sum_{\lambda \in \spe S}\frac{1}{2t}(crs)^{\frac{1-n}{2}}I_{p^+(\lambda)}(\frac{crs}{2t})e^{-\frac{c^2r^2+s^2}{4t}}\phi_0\Phi_{\lambda}(B(p),q)
\end{equation}
In this way we have to calculate $$\lim_{t\rightarrow 0}\int_0^{2}\sum_{\lambda \in \spe S}\frac{1}{2t}(cr^2)^{\frac{1-n}{2}}I_{p^+(\lambda)}(\frac{cr^2}{2t})e^{-\frac{r^2(c^2+1)}{4t}}r^mdr\int_{L_q}\tr(\phi_0\Phi_{\lambda}(B(p),q))dvol_h.$$Clearly $\int_{L_q}\tr(\phi_0\Phi_{\lambda}(B(p),q))dvol_h$ does not depend on $t$ and so, if we label it $\Tr(\tilde{\Phi}_{0,\lambda,q})$, our task now is to calculate $$\lim_{t\rightarrow 0}\int_0^{2}\sum_{\lambda \in \spe S}\frac{1}{2t}(cr^2)^{\frac{1-n}{2}}I_{p^+(\lambda)}(\frac{cr^2}{2t})e^{-\frac{r^2(c^2+1)}{4t}}r^mdr.$$ To do this put $\frac{r^2}{t}=u$. Then $rdr=\frac{tdu}{2}$. Moreover when $r$ goes to 2 $u$ goes to $\frac{4}{t}$ while when $r$ goes to $0$ $u$ goes to zero. So, applying this change of variable, we get
$$\lim_{t\rightarrow 0}\frac{c^{\frac{1-n}{2}}}{4}\int_0^{\frac{4}{t}}e^{-\frac{u(c^2+1)}{4}}\sum_{\lambda\in \spe S}I_{p^+(\lambda)}(\frac{uc}{2})du.$$ Now, by the asymptotic behavior of the integrand, we know that this limit exists and is equal to $$\frac{c^{\frac{1-n}{2}}}{4}\int_0^{\infty}e^{-\frac{u(c^2+1)}{4}}\sum_{\lambda\in \spe S}I_{p^+(\lambda)}(\frac{uc}{2})du.$$So we proved the statement.
\end{proof}

From Theorem \ref{granada} we have the following immediate corollary:
\begin{cor}
\label{santamarta}
In the same hypothesis of Theorem \ref{granada} but without any assumptions about the conic metric $g$ on $reg(X)$ and the metric $\rho$ on $E$. Suppose moreover that  $P^t\circ P:C^{\infty}_c(M,E)\rightarrow C^{\infty}_c(M,E)$ is a generalized Laplacian.
 Then we have the following formula:
\begin{equation}
\label{palmsprings}
L_{2,min}(T)=\sum_{q\in M\cap Fix(f)}\sum_{j=0}^1\frac{(-1)^j\Tr(\phi_j)}{|det(Id-d_qf)|}+
\end{equation} 
$$+\sum_{q\in  sing(X)}\frac{c^{\frac{1-n}{2}}}{4}\int_0^{\infty}e^{-\frac{u(c^2+1)}{4}}\sum_{\lambda\in \spe S}I_{p^+(\lambda)}(\frac{uc}{2})du\Tr(\tilde{\Phi}_{0,\lambda,q})+$$  $$-\sum_{q\in sing(X)}\int_0^{\infty}e^{-\frac{u(c^2+1)}{4}}\sum_{\lambda\in \spe S}I_{p^-(\lambda)}(\frac{uc}{2})du\Tr(\tilde{\Phi}_{1,\lambda,q})$$ where the contribution of the singular points  is calculated fixing any conic metric $g$ on $reg(X)$ and any metric $\rho$ on $E$ which satisfy the assumptions of Theorem \ref{granada}.
\end{cor}

\begin{proof}
As observed in the proof of Theorem \ref{pietrogrado}, by the fact that $P^t\circ P$ is a generalized Laplacian, it follows that  $\mathcal{L}(T,\mathcal{S})$ does not depend on the conic metric we fix on $reg(X)$ and does not depend on the metric $\rho$ we fix on $E$. Therefore, without loss of generality, we can assume that  for each $q\in sing(X)$, using the isomorphism $\chi_q:U_q\rightarrow C_2(L_q)$ of \eqref{derby},  the conic metric $g$ satisfies $g=dr^2+r^2h$ with $h$  that does not depend on $r$ and that each metric $\rho_i$ on $E_i$ does not depend on $r$ in a neighborhood of $\partial \overline{M}$. In this way we are in position to apply Theorem \ref{granada} and therefore  \eqref{palmsprings} follows.
\end{proof}

\section{A thorough analysis of the de Rham case}
\subsection{Applications of the previous results}
As remarked previously, Theorems \ref{sacramento} and  \ref{pietrogrado}, Corollary \ref{vasto}  and in particular \eqref{casablanca} hold for the Hilbert complexes $(L^2\Omega^i(M,g),d_{max/min,i})$. More explicitly, we have the following result:
\begin{teo}
\label{bordeaux}
Let $X$ be a compact and oriented manifold with isolated conical singularities and of dimension $m+1$. Let $g$ be a conic metric over its regular part $reg(X)$. Let $f:X\rightarrow X$ be a map induced by a diffeomorphism $f:\overline{M}\rightarrow \overline{M}$ such that  $f:X\rightarrow X$ fixes each singular point of $X$. Consider $T:=(df)^*\circ f^*$, the natural   endomorphism  of the de Rham complex induced by $f$ . Finally suppose that $f$ has only simple fixed points.  Then we have:
\begin{equation}
\label{osimo}
L_{2,max/min}(T)=\sum_{q\in Fix(f)\cap reg(X)}\sgn det(Id-d_qf)+\mathcal{L}_{max/min}(T,\mathcal{S}).
\end{equation}
If in a neighborhood of each simple fixed point $q$ $f$ satisfies the condition described in \eqref{derby}, then we have:  $L_{2,max/min}(T)=$ 
\begin{equation}
\label{cametrino}
=\sum_{q\in Fix(f)\cap reg(X)}\sgn det(Id-d_qf)+\sum_{q\in  sing(X)}\sum_{i=0}^{m+1}(-1)^i\zeta_{T_i,q}(\Delta_{abs/rel,i})(0)
\end{equation}
where in \eqref{cametrino} the contribution of the singular points is calculated using any conic metric $g$ on $reg(X)$ such that, again through the isomorphism $\chi_q:U_q\rightarrow  C_2(L_q)$ of \eqref{derby}, $g$ takes the form  $dr^2+r^2h$ and $h$ does not depend on $r$.\\
In particular if  each $q\in sing(X)$ is an attractive simple fixed point then we have:  
\begin{equation}
\label{camerino}
L_{2,max/min}(T)=\sum_{q\in Fix(f)\cap reg(X)}\sgn det(Id-d_qf)+\sum_{q\in  sing(X)}L_{2,max/min}(T|_{U_q}).
\end{equation}
while if  each $q\in sing(X)$ is a repulsive  simple fixed point then we have:  
\begin{equation}
\label{camerinone}
L_{2,max/min}(T)=\sum_{q\in Fix(f)\cap reg(X)}\sgn det(Id-d_q(f^{-1}))+\sum_{q\in  sing(X)}L_{2,min/man}(T^*|_{U_q}).
\end{equation}
Moreover in \eqref{cametrino} the member on the right, that is $\mathcal{L}_{max/min}(T,\mathcal{S})$, does not depend on the particular conic metric that we fix on $reg(X)$.
\end{teo}
\begin{proof}
\eqref{osimo}   follows immediately from Theorem \ref{sacramento}. In particular the expression for $\mathcal{L}_{max/min}(T,\mathcal{R})$ follows by a standard argument of linear algebra; see for example \cite{ABL2} or \cite{JR}.  \eqref{cametrino} follows  as in the proof of Theorem \eqref{pietrogrado}; in particular, as remarked in the proof of Lemma \ref{ravenna}, the scaling invariance property for the heat operator associated to positive self-adjoint extension of $\Delta_i$, was proved by  Cheeger in \cite{JC}.  Finally \eqref{camerino} and \eqref{camerinone} follows again from Theorem \ref{pietrogrado}.
\end{proof}

By the assumptions on $f$ it follows that   $f(sing(X))= sing(X)$ and $f(reg(X))= reg(X).$ This implies, see for example \cite{GM}, that  if we fix a perversity $p$ then $f$ induces a well defined map, $f^*$,  between the intersection cohomology groups respect to the  perversity $p$. In particular we have $f^*:I^{\overline{m}}H^i(X)\rightarrow I^{\overline{m}}H^i(X)$ and $f^*:I^{\underline{m}}H^i(X)\rightarrow I^{\underline{m}}H^i(X)$. Therefore it is natural to define in this context, as it is showed in \cite{GMAC}, the \textbf{intersection Lefschetz number} respects to a given perversity $p$ as 
\begin{equation}
\label{montalto}
I^pL(f)=\sum_{i=0}^n\tr(f^*:I^pH^i(X)\rightarrow I^pH^i(X)).
\end{equation}
$I^pL(f)$ is deeply studied, from a topological point of view,  in \cite{GMAC} and \cite{GOMAC} in the more general context of a stratified pseudomanifold; our goal in the next corollaries is to give an analytic description of $I^{\overline{m}}L(f)$ and $I^{\underline{m}}L(f)$ when $X$ is a compact manifold with conical singularities. In particular in  \eqref{wolfsburg} we will give an analytic proof of a formula already proved in \cite{GMAC}.
So, using Theorem \ref{osimo} and Theorem \ref{stia}, we get the following results:
\begin{prop}
\label{abbiategrasso}
In the same hypothesis of Theorem \ref{bordeaux};  let  $q\in  sing(X)$  be  an attractive fixed point . Let $U_q$ be an open neighborhood of $q$ isomorphic to $C_2(L_q)$ and suppose that $f$ satisfies \eqref{derby} and $g$ takes the form $g=dr^2+r^2h$ where $h$ does not depend on $r$.  Then,  for $i<\frac{m+1}{2}$,  we have:
\begin{equation}
\label{escorial}
\Tr((f|_{U_q})^*:H^i_{2,max}(U_q,g|_{U_q})\rightarrow H^i_{2,max}(U_q,g|_{U_q}) )=\Tr(B^*:H^i(L_q)\rightarrow H^i(L_q))
\end{equation} 
\end{prop}

\begin{proof}
 As it is showed in \cite{JEC}, in \eqref{stia} the isomorphism between $H^i_{2,max}(reg(C_2(L_q)),g)$ and $H^i(L_q)$, for $i<\frac{m}{2}+\frac{1}{2}$, is given by the pull-back $\pi^*$ where $\pi:(0,b)\times F\rightarrow F$ is the projection on the second factor and inverse is given by  $v_a$, the evaluation map in $a$,  where $a$ is any  point $(0,2)$. Now by the hypothesis, over $U_q$  $f$ can be written as $(rA(p), B(p))$. An immediate check shows that $\pi^*\circ B^*=B^*\circ \pi^*$ and therefore $\Tr((f|_{U_q})^*)=\Tr(B^*)$. 
\end{proof}

\begin{cor}
\label{compostela}
In the same hypothesis of Theorem \ref{bordeaux}, suppose moreover that near each point $q\in  sing(X)$ $f$ satisfies \eqref{derby}. Then  we have:
\begin{equation}
\label{jerevan}
I^{\underline{m}}L(f)=\sum_{q\in Fix(f)\cap reg(X)}\sgn det(Id-d_qf)+\sum_{q\in  sing(X)}\sum_{i=0}^{m+1}(-1)^i\zeta_{T_i,q}(\Delta_{abs,i})(0)
\end{equation}
and analogously
\begin{equation}
\label{baku}
I^{\overline{m}}L(f)=\sum_{q\in Fix(f)\cap reg(X)}\sgn det(Id-d_qf)+\sum_{q\in  sing(X)}\sum_{i=0}^{m+1}(-1)^i\zeta_{T_i,q}(\Delta_{rel,i})(0)
\end{equation}
Finally, if  $q\in sing(X)$ is an attractive fixed point,  then we have 
\begin{equation}
\label{genga}
\sum_{i=0}^{m+1}(-1)^i\zeta_{T_i,q}(\Delta_{abs,i})(0)=\sum_{i<\frac{m}{2}+\frac{1}{2}}(-1)^i\tr(B^*:H^i(L_q)\rightarrow H^i(L_q))
\end{equation}
and therefore from \eqref{jerevan} we get:
\begin{equation}
\label{wolfsburg}
I^{\underline{m}}L(f)=L_{2,max}(T)=
\end{equation}
$$\sum_{q\in Fix(f)\cap reg(X)}\sgn det(Id-d_qf)+\sum_{q\in  sing(X)}\sum_{i<\frac{m+1}{2}}(-1)^i\Tr(B^*:H^i(L_q)\rightarrow H^i(L_q)).$$

\end{cor}
\begin{proof}
As in Theorem \ref{bordeaux}, to get the Lefschetz numbers, we can use a conic metric $g$ such that, in each neighborhood $U_q$ of $q\in sing(X)$, using the isomorphism $\chi_q:U_q\rightarrow C_2(L_q)$, $g$ takes the form $g=dr^2+r^2h$ where $h$ does not depend on $r$. Now \eqref{jerevan} and \eqref{baku} follow immediately by the previously theorems. Finally \eqref{genga}  and \eqref{wolfsburg} follow  immediately from Proposition \ref{abbiategrasso}.  
\end{proof}
 Finally we have this last corollary; before stating it we recall that a manifold with conical singularities of dimension $m+1$ is a \textbf{Witt} space if $m+1$ is even or, when it is  odd, if $H^{\frac{m}{2}}(L_q)=0$ for each link $L_q$. For more details see, for example, \cite{GM} . 
\begin{cor}
In the same hypothesis of Corollary \ref{compostela}. Suppose moreover that $X$ is a Witt space. Then we get:
\begin{equation}
\label{dakar}
L_{2,max}(T)=L_{2,min}(T),\  \mathcal{L}_{max}(T,\mathcal{S})=\mathcal{L}_{min}(T,\mathcal{S})
\end{equation}
and, if each $q\in sing(X)$ is an attractive fixed point then 
\begin{equation}
\label{cerretodicagli}
\mathcal{L}_{max}(T,\mathcal{S})=\mathcal{L}_{min}(T,\mathcal{S})=\sum_{q\in sing(X)}L_{2,max}(T|_{U_q})=\sum_{q\in sing(X)}L_{2,min}(T|_{U_q})=
\end{equation}
$$=\sum_{q\in sing(X)}\sum_{i<\frac{m+1}{2}}(-1)^i\Tr(B_a^*:H^i(L_q)\rightarrow H^i(L_q)).$$
Finally if each $q\in sing(X)$ is repulsive then we have:
\begin{equation}
\label{udine}
\mathcal{L}_{max}(T,\mathcal{S})=\mathcal{L}_{min}(T,\mathcal{S})=\sum_{q\in sing(X)}L_{2,max}(T^*|_{U_q})=\sum_{q\in sing(X)}L_{2,min}(T^*|_{U_q}).
\end{equation}
\end{cor}
\begin{proof}
\eqref{dakar} follows by the fact that, as it is showed in \cite{JEC}, if $X$ is  a Witt space then for each $i$,  $\Delta_i:\Omega^i_c(reg(X))\rightarrow \Omega^i_c(reg(X))$ is essentially self-adjoint as unbounded operator acting on $L^2\Omega(reg(X),g)$ and this implies that $d_{max,i}=d_{min,i}$ for $i=0,...,m+1$.  \eqref{cerretodicagli} follows by \eqref{dakar} combined with  \eqref{camerino} and \eqref{wolfsburg}. Finally \eqref{udine} follows from the fact that $X$ is Witt and from Theorem \ref{pietrogrado}.
\end{proof}


\subsection{Some further results arising from Cheeger's work on the heat kernel}

The aim of this section is to approach the $L^2-$Lefschetz numbers of the $L^2-$de Rham complex  using the results of Cheeger stated   in \cite{JEC} and in \cite{JC}. For simplicity assume that $X$ is a Witt space.  As recalled previously, if $X$ is  a Witt space and if over $reg(X)$ we put a conic metric, then $\Delta_i:L^{2}\Omega^{*}(reg(X),g)\rightarrow L^2\Omega^*(reg(X),g)$ is essentially self-adjoint for each $i=0,...,m+1$,  with core domain given by the smooth compactly supported forms. In particular this implies that, if $dimX=m+1$, then  for each $i=0,...,m+1$, $d_{max,i}=d_{min,i}$. Therefore, for each map $f:X\rightarrow X$ that induces a geometric endomorphism $T$ as in Theorem \ref{bordeaux}, we have just one $L^2-$Lefschetz number that we label $L_{2}(T)$.\\Now we recall briefly the results we need and we refer to \cite{JEC} and in particular to \cite{JC}, section 3, for the complete details and for the proofs. Let $N$ be an oriented closed manifold of dimension $m$ and let $C(N)$ be the cone over $N$. Endow $reg(C(N))$ with a conic metric $g=dr^2+r^2h$ where $h$ is a riemannian metric over $N$. In the mentioned papers Cheeger introduce four types of differential forms over $reg(C(N))$, called forms of type 1, 2, 3 and 4, such that each eigenform of $\Delta_i$, the Laplacian acting on the $i-$forms over $reg(C(N))$, can be expressed as convergent sum of these forms. For the definition of these forms see \cite{JC} pag. 586-588.

The main reason to introduce these four types of forms is that now we can break the heat operator in four pieces, see \cite{JC} pag. 90-92: $$e^{-t\Delta_i}=\ _{1}e^{-t\Delta_i}+\ _{2}e^{-t\Delta_i}+\ _{3}e^{-t\Delta_i}+\ _{4}e^{-t\Delta_i}$$ where, for each $l=1,...,4$,  $_{l}e^{-t\Delta_i}$ is the heat operator built using the $i-$forms of type $l$. As it is showed in \cite{JC}, pag. 590-592,  it is possible to give an explicit expression for $_{l}e^{-t\Delta_i}.$ In particular for type 1 forms we have:
\begin{equation}
\label{nocera}
_{1}e^{-t\Delta_i}=(r_1r_2)^{a(i)}\sum_j\int_0^{\infty}e^{-t\lambda^2}J_{\nu_{j}(i)}(\lambda r_1)J_{\nu_{j}(i)}(\lambda r_2)\lambda d\lambda\phi^i_{j}(p_1)\otimes \phi^i_j(p_2)= 
\end{equation}
\begin{equation}
\label{cerreto}
=(r_1r_2)^{a(i)}\sum_j\frac{1}{2t}e^{-\frac{r^2_1+r^2_2}{4t}}I_{\nu_j(i)}(\frac{r_1r_2}{2t})\phi^i_j(p_1)\otimes \phi^i_j(p_2)
\end{equation}
where $I_{\nu_{j}(i)}$ is the modified Bessel function (see \cite{MAL} pag. 67), $a(i)=\frac{1}{2}(1+2i-m),\ \nu_{j}(i)=(\mu_j+a^2(i))^{\frac{1}{2}}$  and $a_{j}^{\pm}(i)=a(i)\pm \nu_j(i)$.
The corresponding expression for type 2 forms is 
\begin{equation}
\label{taormina}
_{2}e^{-t\Delta_i}=\sum_jd_1d_2((r_1r_2)^{a(i-1)}\int_{0}^{\infty}e^{-t\lambda^2}J_{\nu_j(i-1)}(\lambda r_1)J_{\nu_j(i-1)}(\lambda r_2)\lambda^{-1}d\lambda \phi_{j}^{i-1}(p_1)\otimes \phi_j^{i-1}(p_2))
\end{equation}
The expression for forms of type 3 is:
\begin{equation}
\label{mugello}
_{3}e^{-t\Delta_i}=\sum_j\int_{0}^{\infty}e^{-t\lambda^2}((-a(i-1)r_1^{a(i-1)}J_{\nu_{j}(i-1)}(\lambda r_1)+r_1^{a(i-1)+1}J'_{\nu_{j}(i-1)}(\lambda r_1)\lambda)\frac{d\phi_j^{i-1}(p_1)}{\sqrt{\mu_j}}
\end{equation}
$$+r_{1}^{a(i-1)-1}J_{\nu_j(i-1)}(\lambda r_1)dr_1\wedge\sqrt{\mu_j}\phi_j^{i-1}(p_1))\otimes((-a(i-1)r_2^{a(i-1)}J_{\nu_j}(\lambda r_2)+$$ $$+r_2^{a(i-1)+1}J'_{\nu_j(i-1)}(\lambda r_2)\lambda)\frac{d\phi_j^{i-1}(p_2)}{\sqrt{\mu_j}}+r_{2}^{a(i-1)-1}J_{\nu_{j}}(\lambda r_2)dr_2\wedge \sqrt{\mu_j}\phi_j^{i-1}(p_2))\lambda^{-1}d\lambda$$
Finally for forms of type 4 we have:
$$_{4}e^{-t\Delta_i}=(r_1r_2)^{a(i-1)}\sum_{j}\int_{0}^{\infty}e^{-t\lambda^2}J_{\nu_j(i-2)}(\lambda r_1)J_{\nu_j(i-2)}(\lambda r_2)\lambda d\lambda dr_1\wedge\frac{d\phi_j^{i-2}(p_1)}{\sqrt{\mu_j}}\otimes dr_2\wedge \frac{d\phi_{j}^{i-2}(p_2)}{\sqrt{\mu_j}}=$$
\begin{equation}
\label{belluno}
=(r_1r_2)^{a(i-2)}\sum_j\frac{1}{2t}e^{-\frac{r_1^2+r_2^2}{4t}}I_{\nu_j(i-2)}(\frac{r_1r_2}{2t})dr_1\wedge\frac{d\phi_j^{i-2}(p_1)}{\sqrt{\mu_j}}\otimes dr_2\wedge \frac{d\phi_{j}^{i-2}(p_2)}{\sqrt{\mu_j}}
\end{equation}
Now suppose that for each point $q\in sing(X)$, over a neighborhood $U_q\cong C_2(L_q)$, $f$  satisfies \eqref{bergamo}.
Using Cheeger's results recalled above,  it make sense to break $T\circ e^{-t\Delta_{i}}$, over $C_2(L_q)$,  as a sum of four pieces such that:
\begin{equation}
\label{taranto}
\lim_{t\rightarrow 0}\Tr(T\circ e^{-t\Delta_{i}})=\lim_{t\rightarrow 0}\Tr(T\circ\  _{1}e^{-t\Delta_{i}}+T\circ\  _{2}e^{-t\Delta_{i}}+T\circ\  _{3}e^{-t\Delta_{i}}+T\circ\ _{4}e^{-t\Delta_{i}}).
\end{equation}
Moreover, using \eqref{tebe}, \eqref{bergamo}, \eqref{cerreto} and  \eqref{belluno}  it is clear that on  $reg(C_2(L_q))$ we have:
\begin{equation}
\label{islaguadalupe}
\tr(T\circ\ _{1}e^{-t\Delta_i})(r,p)=(cr^2)^{a(i)}\sum_j\frac{1}{2t}e^{-\frac{r^2(c^2+1)}{4t}}I_{\nu_j(i)}(\frac{cr^2}{2t})\tr(B^*\phi^i_j\otimes B^*\phi^i_j)
\end{equation}
and analogously
\begin{equation}
\label{miami}
\tr(T\circ\ _{4}e^{-t\Delta_i})(r,p)=(cr^2)^{a(i-2)}\sum_j\frac{1}{2t}e^{-\frac{r^2(c^2+1)}{4t}}I_{\nu_j(i-2)}(\frac{cr^2}{2t})\tr(dr\wedge\frac{d(B^*\phi^{i-2}_j)}{\sqrt{\mu_j}}\otimes dr\wedge \frac{d(B^*\phi^{i-2}_j)}{\sqrt{\mu_j}})
\end{equation}
Now we are in position to state the following result:
\begin{teo}
\label{cannara}
Let $X$, $g$  and $f$ be as in Theorem \ref{bordeaux} such that $dimX=m+1$. Suppose moreover that $X$ is a Witt space and that, on each neighborhood $U_q\cong C_2(L_q)$ of each point $q\in sing(X)$,  $f$ satisfies \eqref{bergamo} and $g$ takes the form $g=dr^2+r^2h$ where $h$  does not depend on $r$. Then, for each $q\in sing(X)$, we have:
\begin{enumerate}
\item The forms of type 1 give a contribution only in degree 0.
\item The contribution given by $q$ in degree zero depends only on the forms of type 1 and we have 
\begin{equation}
\label{ulanbator}
\zeta_{T_0,q}(\Delta_0)(0)=\frac{c^{\frac{1-m}{2}}}{4}(\int_{0}^{\infty}e^{-u(c^2+1)} \sum_jI_{\nu_j(0)}(\frac{cu}{2})du)(\Tr(B^*\phi^i_j\otimes B^*\phi^i_j))
\end{equation}
\item The forms of type 4 give a contribution only in degree 2 and this contribution is 
\begin{equation}
\label{bengasi}
\Tr(T_2\circ\ _{4}e^{-t\Delta_2}) =\frac{c^{\frac{1-m}{2}}}{4}(\int_{0}^{\infty}e^{-\frac{u(c^2+1)}{4}}\sum_jI_{\nu_j(0)}(\frac{cu}{2})du)(\Tr(dr\wedge\frac{d(B^*\phi^{i-2}_j)}{\sqrt{\mu_j}}\otimes dr\wedge \frac{d(B^*\phi^{i-2}_j)}{\sqrt{\mu_j}}))
\end{equation}
where $\Tr(T_2\circ\ _{4}e^{-t\Delta_2})$ is taken over $reg(C_2(L_q))$.
\item The contribution given by $q$ in the others degrees, that is $i\neq 0,2$, depends only on the forms of type 2 and 3. 
\end{enumerate}
\end{teo}

\begin{proof}
First of all we note that from \eqref{cerreto}, \eqref{taormina}, \eqref{mugello} and \eqref{belluno}  it follows that $_{1}e^{-t\Delta_i}=e^{-t\Delta_i}$  for $i=0$ and that $_{4}e^{-t\Delta_i}$ occurs only for $i\geq 2$. Now,
using \eqref{cerreto} and \eqref{islaguadalupe} we know that, over $reg(C_2(L_q))$, $$\lim_{t\rightarrow 0}\Tr(T_i\circ\ _{1}e^{-t\Delta_i})=\lim_{t\rightarrow 0}\int_{0}^2\int_{L_q}(cr^2)^{a(i)}\sum_j\frac{1}{2t}e^{-\frac{r^2(c^2+1)}{4t}}I_{\nu_j(i)}(\frac{cr^2}{2t})\tr(B^*\phi^i_j\otimes B^*\phi^i_j)r^mdrdvol_h.$$ Clearly this last term it is in turn equal to 
\begin{equation}
\label{cartagine}
\lim_{t\rightarrow 0}((\int_0^2(cr^2)^{a(i)}\frac{1}{2t}e^{-\frac{r^2(c^2+1)}{4t}}\sum_jI_{\nu_j(i)}(\frac{cr^2}{2t})r^mdr)(\Tr(B^*\phi^i_j\otimes B^*\phi^i_j)))
\end{equation}
and therefore, to get the first two points we have to calculate
\begin{equation}
\label{lubiana}
\lim_{t\rightarrow 0}\int_0^2(cr^2)^{a(i)}\frac{1}{2t}e^{-\frac{r^2(c^2+1)}{4t}}\sum_jI_{\nu_j(i)}(\frac{cr^2}{2t})r^mdr
\end{equation}

First of all remember that $a(i)=\frac{1}{2}(1-m+2i)$; therefore $r^{2a(i)}r^m=r^{2i+1}$. Now put $\frac{r^2}{t}=u$.  It follows immediately that $dr=\frac{tdu}{2r}$. Now, by the fact that $r^2=tu$ it follows that $r^{2i+1}=t^iu^ir$ and therefore we also get $r^{2i+1}dr=\frac{t^{i+1}u^idu}{2}$. Moreover when $r$ goes to $2$ then $u$ goes to $\frac{2}{t}$ and when $r$ goes to $0$ then $u$ goes to $0$. In this way we have 
\begin{equation}
\label{canossa}
\lim_{t\rightarrow 0}\frac{c^{a(i)}}{4}t^i\int_0^{\frac{2}{t}}e^{-u(c^2+1)}\sum_jI_{\nu_j(i)}(\frac{c^2u}{2})u^idu
\end{equation}
Now, by the asymptotic behavior of the integrand it follows that $$\lim_{t\rightarrow 0}\frac{c^{a(i)}}{4}\int_0^{\frac{2}{t}}e^{-u(c^2+1)}\sum_jI_{\nu_j(i)}(\frac{c^2u}{2})u^idu=\frac{c^{a(i)}}{4}\int_0^{\infty}e^{-u(c^2+1)}\sum_jI_{\nu_j(i)}(\frac{c^2u}{2})u^idu.$$
Therefore we can conclude that 
\begin{equation}
\label{anagni}
\eqref{canossa}=\left\{
\begin{array}{ll}
\frac{c^\frac{1-m}{2}}{4}\int_0^{\infty}e^{-u(c^2+1)}\sum_jI_{\nu_j(0)}(\frac{c^2u}{2})du & i=0\\
0 & i>0
\end{array}
\right.
\end{equation} 

In this way we proved the first and the second assertion. For the third statement the proof is completely analogous to the previous one. Also in this case it is clear that in order to establish the assertion we have to calculate: $$\lim_{t\rightarrow 0}c^{a(i-2)+1}\int_0^2\frac{1}{2t}e^{-\frac{r^2(c^2+1)}{4t}}\sum_jI_{\nu_j(i-2)}(\frac{cr^2}{2t})r^{2i-3}dr.$$ Now if we put again $\frac{r^2}{t}=u$ the remaining part of the proof is completely analogous to that one of the first two points.\\Finally the last point follows from the first three points. 
\end{proof}

\end{document}